\documentclass[a4paper,12pt,reqno]{amsart}
\usepackage[width=6.2in]{geometry}
\usepackage{amsmath,amsfonts,amsthm,amssymb,color}
\usepackage[latin1]{inputenc}

\newcommand{\der}{\delta}
\newcommand{\Qha}{\Hat{\mathcal{Q}}}

\newcommand{\cacha}{\Hat{\mathcal{C}}}

\newcommand{\delha}{\hat{\delta}}

\newcommand{\Laha}{\hat{\Lambda}}

\newcommand{\norm}[1]{\lVert #1\rVert}

\newcommand{\yti}{\tilde{y}}
\newcommand{\zti}{\tilde{z}}

\newcommand{\cB}{\mathcal{B}}

\newcommand{\ka}{\kappa}

\newcommand{\ott}{[0,T]}

\newcommand{\cfx}{c_{f,\textbf{X}}}

\DeclareMathOperator{\id}{\text{Id}}

\makeatletter
\newcommand{\eqcolon}{\mathrel{\mathord{=}\raise.2\p@\hbox{:}}}
\newcommand{\coloneq}{\mathrel{\raise.2\p@\hbox{:}\mathord{=}}}
\makeatother

\newcommand{\RR}{\mathbb{R}}

\newcommand{\CC}{\mathcal{C}}

\newcommand{\R}{\mathbb R}
\newcommand{\N}{\mathbb N}

\newcommand{\cb}{\mathcal B}
\newcommand{\cac}{\mathcal C}

\newcommand{\cj}{\mathcal J}
\newcommand{\cl}{\mathcal L}
\newcommand{\cn}{\mathcal N}
\newcommand{\cq}{\mathcal Q}
\newcommand{\cx}{\mathcal X}
\newcommand{\cs}{\mathcal S}

\newcommand{\cp}{\mathcal P}

\newcommand{\al}{\alpha}
\newcommand{\ep}{\varepsilon}

\newcommand{\ga}{\gamma}

\newcommand{\si}{\sigma}

\newcommand{\vp}{\varphi}

\newcommand{\lp}{\left(}
\newcommand{\rp}{\right)}
\newcommand{\lc}{\left[}
\newcommand{\rc}{\right]}
\newcommand{\lcl}{\left\{}
\newcommand{\rcl}{\right\}}
\newcommand{\lln}{\left|}
\newcommand{\rrn}{\right|}

\newcommand{\bean}{\begin{eqnarray*}}
\newcommand{\eean}{\end{eqnarray*}}
\newcommand{\ben}{\begin{enumerate}}
\newcommand{\een}{\end{enumerate}}
\newcommand{\beq}{\begin{equation}}
\newcommand{\eeq}{\end{equation}}

\newtheorem{theorem}{Theorem}[section]

\newtheorem{definition}[theorem]{Definition}
\newtheorem{notation}[theorem]{Notation}

\newtheorem{hypothesis}{Hypothesis}
\newtheorem{lemma}[theorem]{Lemma}

\newtheorem{proposition}[theorem]{Proposition}

\theoremstyle{remark}
\newtheorem{remark}[theorem]{Remark}

\begin{document}

\title{Non-linear Rough Heat Equations}
\date{\today}



\author{A. Deya}
\author{M. Gubinelli}
\author{S. Tindel}
\address[A.~Deya, S.~Tindel]{Institut {\'E}lie Cartan Nancy\\ Universit\'e de Nancy\\ B.P. 239,
54506 Vand{\oe}uvre-l{\`e}s-Nancy Cedex, France
}
\email{deya@iecn.u-nancy.fr, tindel@iecn.u-nancy.fr}

\address[M.~Gubinelli]{CEREMADE\\
Universit\'e de Paris-Dauphine\\
75116 Paris, France}

\email{massimiliano.gubinelli@ceremade.dauphine.fr}
\thanks{This research is supported by the ANR Project ECRU - Explorations on rough paths}

\begin{abstract}
This article is devoted to define and solve an evolution equation of the form $dy_t=\Delta y_t \, dt+ dX_t(y_t)$, where $\Delta$ stands for the Laplace operator on a space of the form $L^p(\R^n)$, and $X$ is a finite dimensional noisy nonlinearity whose typical form is given by $X_t(\varphi)=\sum_{i=1}^N  \, x^{i}_t f_i(\varphi)$, where each $x=(x^{(1)},\ldots,x^{(N)})$ is a $\ga$-Hölder function generating a rough path and each $f_i$ is a smooth enough function defined on $L^p(\R^n)$. The generalization of the usual rough path theory allowing to cope with such kind of systems is carefully constructed.
\end{abstract}


\maketitle

\section{Introduction}
The rough path theory, which was first formulated in the late 90's by Lyons~\cite{Lyons, LyonsBook} and then reworked by various authors \cite{FV, rough}, offers a both elegant and efficient way of defining integrals driven by rough signal. This pathwise approach enables the interpretation and resolution of the standard (rough) differential system
\begin{equation}\label{syst-class}
dy_t=\si(y_t) \, dx_t \quad , \quad y_0=a,
\end{equation}
where $x$ is only a Hölder process, and also the treatment of less classical (rough) differential systems such that the delay equation \cite{NNT} or the integral Volterra systems \cite{DT, DTbis}. In all of those situations, the fractional Brownian motion stands for the most common process for which the additional hypotheses required during the construction are actually satisfied.

\smallskip

In the last few years, several authors provided some kind of similar pathwise treatment for quasi-linear equations associated to non-bounded operators, that is to say of the rather general form 
\begin{equation}\label{syst-gen}
dy_t=Ay_t \, dt+ dX_t(y_t), \qquad t\in\ott
\end{equation}
where $T$ is a strictly positive constant, $A$ is a non-bounded operator defined on a (dense) subspace of some Banach space $V$ and $X\in \CC([0,T]\times V; V)$ is a noise which is irregular in time and which evolves in the space of vectorfields acting on the Banach space at stake. Their results apply in particular to some specific partial differential equations perturbed by samples of (infinite-dimensional) stochastic processes. 

\smallskip

To our knowledge, two different approaches have been used to tackle the issue of giving sense to (\ref{syst-gen}):
\begin{itemize}
\item The first one essentially consists in returning to the usual formulation (\ref{syst-class}) by means of tricky transformations of the initial system (\ref{syst-gen}). One is then allowed to resort to the numerous results established in the standard background of rough paths analysis. 
As far as this general method is concerned, let us quote the work of Caruana and Friz \cite{CF}, Caruana, Friz and Oberhauser~\cite{CFH}, as well as the promising approach of Teichmann \cite{Tei}.
\item The second approach is due to the last two authors of the present paper, and is based on a  formalism which combines (analytical) semigroup theory and rough paths methods. This formulation can be seen as a ``twisted" version of the classical rough path theory.
\end{itemize}
Of course, one should also have in mind the huge literature concerning the case of evolution equations driven by usual Brownian motion, for which we refer to \cite{DPZ} for the infinite dimensional setting and to \cite{Da} for the multiparametric framework. In the particular case of the stochastic heat equation driven by an infinite dimensional Brownian motion, some sharp existence and uniqueness results have been obtained in \cite{PZ} in a Hilbert space context, and in \cite{BE} for Banach valued solutions (closer to the situation we shall investigate). In the Young integration context, some recent efforts have also been made in order to define solutions to parabolic \cite{Maslo-Nualart,glt} or wave type \cite{qt} equations. We would like to mention also the application of rough path ideas to the solution of dispersive equation (both deterministic and stochastic) with low-regularity initial conditions~\cite{g-kdv}.

\smallskip

The present article goes back to the setting we have developed in \cite{GT}, and proposes to fill two gaps left by the latter paper. More specifically, we focus (for sake of clarity) on the case of the heat equation in $\R^n$ with a non-linear fractional perturbation, and our aim is to give a reasonable sense and solve the equation
\begin{equation}\label{eq:heat-intro}
dy_t =\Delta y_t \, dt+ dX_t(y_t),
\end{equation}
where $\Delta$ is the Laplacian operator considered on some $L^p(\R^n)$ space (with $p$ chosen large enough and specified later on), namely 
$$
\Delta: D(\Delta) \subset L^p(\R^n) \to L^p(\R^n).
$$
Then the first improvement we propose here consists in considering a rather general noisy nonlinearity $X$ evolving in a Hölder space  $\CC^\ga(L^p(\R^n);L^p(\R^n))$, with $\ga<1/2$, instead of the polynomial perturbations we had in \cite{GT}. A second line of generalization is that we also show how to push forward the rough type expansions in the semi-group context, and will be able to get some existence and uniqueness results up to $\ga>1/4$, instead of $\ga>1/3$.

\smallskip

As usual in the stochastic evolution setting, we study equation (\ref{eq:heat-intro}) in its mild form, namely: 
\begin{equation}\label{equation-forme-faible}
y_t=S_t y_0+\int_0^t S_{t-s} dX_s(y_s),
\end{equation}
where $S_t: L^p(\R^n) \to L^p(\R^n)$ designates the heat semigroup on $\R^n$. This being said, and before we state an example of the kind of result we have obtained, let us make a few remarks on the methodology we have used. 

\smallskip

\noindent
\textbf{(a)} 
The main price to pay in order to deal with a general nonlinearity is that we only consider a finite dimensional noisy input. Namely, we stick here to a noise generated by a $\ga$-Hölder path $x=(x^{(1)},\ldots,x^{(N)})$ and evolving in a finite-dimensonal subspace of $\CC(L^p(\R^n);L^p(\R^n))$, which can be written as:
\begin{equation}\label{forme-bruit}
X_t(\varphi)=\sum_{i=1}^N  \, x^{i}_t f_i(\varphi),
\end{equation}
with some fixed elements $\{f_i\}_{i=1,\dots,N}$ of $\CC(L^p(\R^n);L^p(\R^n))$, chosen of the particular form
$$
f_i(\varphi)(\xi) = \sigma_i(\xi,\varphi(\xi))
$$
for sufficiently smooth functions $\sigma_i : \RR^n \times \RR \to \RR$.

\smallskip

Note that the hypothesis of a finite-dimensional noise is also assumed in \cite{CF} or \cite{Tei}. Once again, our aim in \cite{GT} was to deal with irregular homogeneous noises in space, but we were only able to tackle the case of a linear or polynomial dependence on the unknown. As far as the form of the nonlinearity is concerned, let us mention that \cite{CF} deals with a linear case, while the assumptions in \cite{Tei} can be read in our setting as: one is allowed to define an extended function $\tilde f_i(\varphi):=S_{-t}f_i(S_t \vp)$, which is still a smooth enough function of the couple $(t,\vp)$. As we shall see, the conditions we ask in the present article for $f_i$ are much less stringent, and we shall recover partially the results of \cite{Tei} at Section \ref{sec:glob-sol}. 

\smallskip

\noindent
\textbf{(b)} In order to interpret (\ref{equation-forme-faible}), the reasoning we will resort to is largely inspired by the analysis of the standard rough integrals. For this reason, let us recall briefly the main features of the theory, as it is presented in \cite{rough}: the interpretation of $\int y_s \, dx_s$ (with $x$ a finite-dimensional irregular noise) stems from some kind of \textit{dissection} of the usual Riemann-Lebesgue integral $\int y \, d\tilde{x}$, when $\tilde{x}$ is a regular driving process. This work of dismantling appeals to two recurrent operators acting on spaces of $k$-variables functions ($k \geq 1$): the increment operator $\der$ and its potential inverse, the sewing map $\Lambda$, the existence of which hinges on some specific regularity conditions. If $y$ is a $1$-variable function, then $\der$ is simply defined as $(\der y)_{ts}=y_t-y_s$, while if $z_{ts}=\int_s^t (y_t-y_u) \, d\tilde{x}_u$, then $(\der z)_{tus}=(\der y)_{tu}(\der \tilde{x})_{us}$. With such notations, one has for instance
$$\int_s^t y_u \, d\tilde{x}_u=\lp \int_s^t d\tilde{x}_u \rp y_s+\int_s^t (y_t-y_u) \, d\tilde{x}_u=\lp \int_s^t d\tilde{x}_u \rp y_s+\lp \der^{-1} \lp (\der y)(\der \tilde{x}) \rp \rp_{ts}.$$
Of course, the latter equality makes only sense once the invertibility of $\der$ has been justified.

\smallskip

During the process of dissection, it early appears, and this is the basic principles of the rough path theory, that in order to give sense to $\int y_s \, dx_s$, it suffices to justify the existence of the iterated integrals associated to $x$: $x^1_{ts}=\int_s^t dx_u$, $x^2_{ts}=\int_s^t dx_u \int_s^u dx_v$, etc., up to an order which is linked to the Hölder regularity of $x$. If $x$ is $\ga$-Hölder for some $\ga >1/2$, then only $x^1$ is necessary, whereas if $\ga \in (1/3,1/2)$, then $x^2$ must come into the picture.

\smallskip

Once the integral has been defined, the resolution of the standard system
\begin{equation}\label{dim-finie}
(\der y)_{ts}=\int_s^t \si(y_u) \, dx_u \quad , \quad y_0=a,
\end{equation}
where $\si$ is a regular function, is quite easy to settle by a fixed-point argument.

\smallskip

\noindent
\textbf{(c)} As far as (\ref{equation-forme-faible}) is concerned, the presence of the semigroup inside the integral prevents us from writing this infinite-dimensional system under the general form (\ref{dim-finie}). If $y$ is a solution of (\ref{equation-forme-faible}) (suppose such a solution exists), its variations are actually governed by the equation (let $s <t$)
$$(\der y)_{ts}=y_t-y_s=S_t y_0-S_s y_0+\int_0^s \lc S_{t-u}-S_{s-u}\rc \, dX_u(y_u)+\int_s^t S_{t-u} \, dX_u(y_u),$$
which, owing to the additivity property of the semigroup, reduces to
\begin{eqnarray}\label{transfo-semigroupe}
(\der y)_{ts}&=& a_{ts}y_s+\int_s^t S_{t-u} \, dX_u(y_u),
\end{eqnarray}
where $a_{ts}=S_{t-s}-\id$. Here occurs the simple idea of replacing $\der$ with the new coboundary operator $\delha$ defined by $(\delha y)_{ts}=(\der y)_{ts}-a_{ts}y_s$. Equation (\ref{transfo-semigroupe}) then takes the more familiar form
\begin{equation}\label{systeme-final}
(\delha y)_{ts}=\int_s^t S_{t-u} dX_u(y_u) \quad , \quad y_0=\psi.
\end{equation}

In the second section of the article, we will see that the operator $\delha$, properly extended to act on $k$-variables functions ($k\geq 1$), satisfies properties analogous to $\der$. In particular, the additivity property of $S$ enables to retrieve the cohomology relation $\delha \delha$, which is at the core of the most common constructions based on $\der$. For sake of consistence, we shall adapt the notion of regularity of a process to this context: a $1$-variable function will be said to be $\ga$-Hölder \textit{in the sense of $\delha$} if for any $s,t$, $|(\delha y)_{ts}| \leq c \lln t-s \rrn^\ga$. It turns out that the properties of $\delha$ suggests the possibility of inverting $\delha$ through some operator $\Laha$, just as $\Lambda$ inverts $\der$. This is the topic of Theorem \ref{existence-laha}, which was the starting point of \cite{GT} and also the cornerstone of all our present constructions.

\smallskip

\noindent
\textbf{(d)}
Sections 3 and 4 will then be devoted to the interpretation of the integral appearing in~(\ref{systeme-final}). To this end, we will proceed as with the standard system (\ref{dim-finie}), which means that we will suppose at first that $X$ is regular in time and under this hypothesis, we will look for a decomposition of the integral in terms of "iterated integrals" depending only on $X$. For some obvious stability reasons, it matters that the dissection mainly appeal to the operators $\delha$ and $\Laha$. 

\smallskip

However, in the course of the reasoning, it will be necessary to control the regularity in time of the function $u\mapsto f_i(y_u)$ as a function of the regularity of $y$. To do so, one can only resort to the tools of standard differential calculus, based on the Taylor formula. Unfortunately, those methods can't take our $\delha$-formalism into consideration. For instance, it would be futile to search for an equivalent of the rule 
\begin{equation}\label{eq:taylor-intro}
\der(f_i(y))_{ts}(\xi)=\int_0^1 dr \, \sigma_i'(\xi,y_s(\xi)+r(\der y)_{ts}(\xi))(\der y)_{ts}(\xi),
\end{equation}
which should be expressed in terms of $\delha$ only. This obvious remark obliges us to alternate the use of the two operators $\der$ and $\delha$, but the procedure raises some issues as far as Hölder regularity is concerned. Indeed, a function which is $\ga$-Hölder in the classical sense, that is in the sense of $\der$, is not necessarily $\ga$-Hölder in the sense of $\delha$. In such a situation, if we refer to the definition of $\delha$ ($(\delha y)_{ts}=(\der y)_{ts}-a_{ts}y_s$), we would like to retrieve $\lln t-s \rrn^\ga$-increments from the operator $a_{ts}$ itself. This can be done by letting the fractional Sobolev spaces come into play. Namely, wet set $\cb=L^p(\R^n)$ and for $\al \in [0,1/2)$, we also write $\cb_{\al,p}$ for the fractional Sobolev space of order $\al$ based on $\cb$ (the definition will be elaborated on in Section \ref{section-cadre}). One can then resort to the relation (see Section \ref{section-cadre})
$$\text{if} \, \varphi \in \cb_{\al,p}, \ \norm{a_{ts} \varphi}_{\cb_p} \leq c \lln t-s \rrn^\al \norm{\varphi}_{\cb_{\al,p}}.$$
Of course, we will have to pay attention to the fact that this time regularity gain occurs to the detriment of the spatial regularity. It is also easily conceived that we will require $\cb_{\al,p}$ to be an algebra of continuous functions, which explains why we work in some $L^p$ spaces with $p$ large enough.

\smallskip

The difficulties evoked by equation (\ref{eq:taylor-intro}) are specific to the non-linear case. If the vectorfields $\{f_i\}_{i=1,\dots,N}$ are linear, then we don't need any recourse to the Taylor formula and the decomposition can be written thanks to $\delha$ and $\Laha$ only. This particular case has been dealt with in \cite{GT}, as well as the polynomial case, for which we suggested a treatment based on trees-indexed integral~\cite{g-ramif,g-abs}. In our situation, we shall see that the landmarks of the construction, that is to say the counterparts of the usual step-2 rough path $(\int dx, \iint dx \otimes dx)$, are (morally) some operators acting on $\cb$, defined as follows: for $\vp,\psi\in\cb$, set
\begin{equation}\label{op-intro-1}
X^{x,i}_{ts}(\vp)=\int_s^t S_{tu}(\vp) \, dx^i_u \quad , \quad 
X^{xa,i}_{ts}(\vp,\psi)=\int_s^t S_{tu}  \lc  a_{us}(\vp)\cdot \psi \rc \, dx^i_u,
\end{equation}
\begin{equation}\label{op-intro-2}
X^{xx,ij}_{ts}(\vp)=\int_s^t S_{tu}(\vp)  \, \der x^j_{us}\,  dx^i_u  ,
\end{equation}
for $i,j=1,\dots,N$,
where $\vp\cdot\psi$ is the pointwise multiplication operator of $\vp$ by $\psi$.

\smallskip

In a quite natural way, the results established in Section \ref{section-cas-young} by using development at first order only, will be applied to a $\ga$-Hölder process $x$ with $\ga>1/2$. The considerations of Section \ref{section-cas-rough}, which involve more elaborate developments, will then enable the treatment of the case $1/3 <\ga \leq 1/2$. Finally a few words will be said about the case $\ga \in (1/4,1/3]$ in Section \ref{section-rough-case-2}, and we shall see how the stack of operators allowing the rough path analysis grows at order 3.

\smallskip

It is also crucial to see how our theory applies  to concrete situations. To this purpose,
using an elementary integration by parts argument, we will see in Section \ref{section-appli-mbf-13} that in order to define the operators given by (\ref{op-intro-1}) and (\ref{op-intro-2}) properly, the additional assumptions on $x$ reduce to the standard rough-paths hypotheses. In this way, the results of this article can be applied to a $N$-dimensional fractional Brownian motion $x$ with Hurst index $H>1/4$, thanks to the previous works of Coutin-Qian \cite{CQ} or Unterberger \cite{Un}. This also means that in the end, the solution to the rough PDE (\ref{eq:heat-intro}) is a continuous function of the initial condition and $x^1,x^2,x^3$, which suggests (as \cite{Tei} does) that one can also solve the noisy heat equation by means of a variant of the classical rough path theory. However, we claim that our construction is really well suited for the evolution equation setting, insofar that the arguments developed here can be extended naturally to an infinite dimensional noise, at the price of some more intricate technical considerations. We plan go back to this issue in a further publication.

\smallskip

With all these consideration in mind, we can now give an example of the kind of result which shall be obtained in the sequel of the paper (given here in the first non trivial rough case for $X$, that is a Hölder continuity exponent $1/3<\ga\le 1/2$):
\begin{theorem}
Let $X$ be a noisy nonlinearity of the form (\ref{forme-bruit}), where:

\noindent
(i) The noisy part $x$ is a $N$-dimensional Hölder-continuous signal in $\cac^\ga(\ott;\R^N)$ for a given $\ga>1/3$. We also assume that $x$ allows to define a Levy area $x^2$ in the sense given by Hypothesis \ref{hypo-x-rough-path}.

\noindent
(ii) Each nonlinearity $f_i$ can be written as $[f_i(\varphi)](\xi) = \sigma_i(\xi,\varphi(\xi))$, where the function $\si_i:\R^n\times\R\to\R$ is such that $\si_i(\cdot,\eta)=0$ outside of a ball $B_{\R^N}(0,M)$, independently of $\eta\in\R$. We also ask $\eta\mapsto\si_i(\xi,\eta)$ to be a $C_b^3(\R)$ function for all $\eta\in\R^N$.

Then equation (\ref{equation-forme-faible}) admits a unique solution $y$ on an interval $\ott$, for a strictly positive time $T$ which depends on $x$ and $x^2$. Furthemore, the solution $y$ is a continuous function of $(y_0,x,x^2)$.
\end{theorem}
Notice that this theorem is directly applicable to the fractional Brownian setting for $H>1/3$. The case of a Hölder coefficient $1/4<\ga\le 1/3$ is also discussed at the end of the article. 

\smallskip

Here is how our paper is structured: Section \ref{section-cadre} is devoted to recall somme basic facts about algebraic integration with respect to a semi-group of operators, taken mainly from~\cite{GT}. Then we deal with the easy case of Young integration at Section \ref{section-cas-young}. This allows to solve equations for a noisy input with any Hölder continuity exponent $\ga>1/2$ (recall that we had to consider $\ga>5/6$ in \cite{GT}), and it should also be mentioned that we obtain a global solution for the RPDE (\ref{equation-forme-faible}) in this case. The first rough case, that is a Hölder continuity exponent $\ga\in(1/3,1/2]$, is handled at Section \ref{section-cas-rough}. Observe that the abstract results obtained there are expressed in terms of  the operators $X^x,X^{xa}$ and $X^{xx}$ defined at equation (\ref{op-intro-1}) and (\ref{op-intro-2}). It is also important to notice that only local solutions are obtained in the general case, due to the fact that our nonlinearity cannot be considered as a bounded function on the Sobolev spaces $\cb_{\al,p}$. We will show however at Section \ref{sec:glob-sol} that considering a smoothed version of the nonlinearity, a global solution to equation (\ref{equation-forme-faible}) can be constructed. Section \ref{section-appli-mbf-13} is then devoted to the translation of these results in terms of $x^1$ and $x^2$ by a simple integration by parts argument, and thus to the application of the abstract results to concrete situations. Finally, we discuss in Section \ref{section-rough-case-2} the rougher case of a Hölder continuity exponent of the noise $x$ satisfying $1/4<\ga\le 1/3$.

\section{Algebraic integration associated to the heat semigroup}\label{section-cadre}

This first section aims at introducing the framework of our study, as well as the different tools evoked in the introduction. The main point here is the definition and the basic properties of the infinite-dimensional coboundary operator $\delha$ already alluded to in the introduction, together with its inverse $\Laha$. At first, we will recall some elementary properties of the heat semigroup, which will actually be used in the construction of $\Laha$ (Theorem \ref{existence-laha}). 
\subsection{Framework}

We will focus on the case of the heat equation on $L^p(\R^n)$, for some integrer $p$ that will be precised during the study. We denote by $\Delta=\Delta_p$ the Laplacian operator, considered on the (classical) Sobolev space $W^{2,p}(\R^n)$, and by $S_t$ the associated heat semigroup, which is also defined by the convolution

\begin{equation}\label{semigroupe-chaleur}
S_t \varphi=g_t \ast \varphi \quad , \quad \text{with} \ g_t(\xi)=\frac{1}{(2\pi t)^{n/2}} e^{-|\xi|^2/2t}.
\end{equation}

As explained at point \textbf{(d)} of the introduction, the interplay between the linear and the non-lienar part of the equation invites us to let the fractional Sobolev spaces come into the picture:

\begin{notation}
For any $\al >0$, for any $p\in \N^\ast$, we will denote by $\cb_{\al,p}$ the space $(\id -\Delta)^{-\al}(L^p(\R^n))$, endowed with the norm
$$\norm{\vp}_{\cb_{\al,p}}=\norm{\vp}_{L^p(\R^n)}+\norm{(-\Delta)^{\al}\vp}_{L^p(\R^n)}. $$
Set also $\cb_p=\cb_{0,p}=L^p(\R^n)$ for any $p\in \N^\ast \cup \{ \infty \}$.
\end{notation}

The space $\cb_{\al,p}$ is also refered to as the \textit{Bessel potential} of order $(2\al,p)$. Adams (\cite{Ad}) or Stein (\cite{Ste}) gave a thorough description of those fractional Sobolev spaces. Let us indicate here the two properties that we will resort to in the applications:

\begin{itemize}
\item \textit{Sobolev inclusions}: If $0\leq \mu \leq 2\al-\frac{n}{p}$, then $\cb_{\al,p}$ is continuously included in the space $\cac^{0,\mu}(\R^n)$ of the bounded, $\mu$-Hölder functions.
\item \textit{Algebra}: If $2\al p > n$, then $\cb_{\al,p}$ is a Banach algebra, or in other words $\norm{\varphi \cdot \psi}_{\cb_{\al,p}} \leq \norm{\varphi}_{\cb_{\al,p}} \norm{\psi}_{\cb_{\al,p}}$.
\end{itemize}

\smallskip

The general theory of fractional powers of operators then provides us with sharp estimates for the semigroup $S_t$ (see for instance \cite{paz} or \cite{engel00a}):
\begin{proposition}
Fix a time $T>0$. $S_t$ satisfies the following properties:
\begin{itemize}
\item \textit{Contraction}: For all $t \geq 0$, $\al \geq 0$, $S_t$ is a contraction operator on $\cb_{\al,p}$.
\item \textit{Regularization}: For all $t \in (0,T]$, $\al\geq 0$, $S_t$ sends $\cb_p$ on $\cb_{\al,p}$ and 
\begin{equation} \label{regu-prop-semi}
\norm{S_t \varphi}_{\cb_{\al,p}} \leq c_{\al,T}\, t^{-\al} \norm{\varphi}_{\cb_p}.
\end{equation}
\item \textit{Hölder regularity}. For all $t\in (0,T]$, $\varphi \in \cb_{\al,p}$,
\begin{equation}\label{regu-hold-semi}
\norm{S_t\varphi-\varphi}_{\cb_p} \leq c_{\al,T}\, t^\al \norm{\varphi}_{\cb_{\al,p}}.
\end{equation}
\begin{equation}\label{regu-hold-semi-2}
\norm{\Delta S_t \vp}_{\cb_p} \leq c_{\al,T}\,  t^{-1+\al} \norm{\varphi}_{\cb_{\al,p}}.
\end{equation}
\end{itemize}
\end{proposition}

At some point of our study, the interpretation of the integral $\int_s^t S_{tu} \, dx^i_u \, f_i(y_u)$ will require a Taylor expansion of the (regular) function $f_i$. As a result, pointwise multiplications of elements of $\cb_p$ are to appear, giving birth to elements of $\cb_{p/k}$, $k\in \lcl 1, \ldots, p \rcl$. In order to go back to the base space $\cb_p$, we shall resort to the following additional properties of $S_t$, which accounts for our use of the spaces $\cb_p$ ($p\geq 2$) instead of the classical Hilbert space $\cb_2$:

\begin{proposition}
For all $t>0$, $k\in \lcl 1,\ldots,p \rcl$, $\vp \in \cb_{p/k}$, one has
\begin{equation}\label{p/k-p}
\norm{S_t \vp }_{\cb_p} \leq c_{k,n} t^{-\frac{n(k-1)}{2p} } \norm{\vp }_{\cb_{p/k}},
\end{equation}
\begin{equation}
\norm{AS_t \vp }_{\cb_p} \leq c_{k,n} t^{-1-\frac{n(k-1)}{2p} } \norm{\vp }_{\cb_{p/k}}.
\end{equation}
\end{proposition}

\begin{proof}
Those are direct consequences of the Riesz-Thorin theorem. Indeed, for any $\vp \in \cb_{p/k}$,
$$\norm{S_t \vp }_{\cb_p} \leq \norm{g_t \ast \vp}_{\cb_p} \leq \norm{g_t}_{\cb_{p/(p-k+1)}} \norm{\vp }_{\cb_{p/k}} \leq c_{k,n} t^{-\frac{n(k-1)}{2p} } \norm{\vp}_{\cb_{p/k}}.$$
The second inequality can be proved in the same way, since $AS_t \vp=\lp \frac{dS_t}{dt} \rp \vp=\partial_t g_t \ast \vp$.
\end{proof}

Let us finally point out the following result of Strichartz \cite{stri}, which will be at the core of our fixed-point argument through Proposition \ref{prop-stri} (see also \cite{hir} for more general results):

\begin{proposition}\label{result-stri}
For all $\al \in (0,1/2)$, for all $p>1$, set
$$T_\al f(\xi)=\lp \int_0^1 r^{-1-4\al} \lc \int_{\lln \eta \rrn \leq 1} \lln f(\xi+r\eta)-f(\xi) \rrn \, d\eta \rc^2 \, dr \rp^{1/2}.$$
Then $f\in \cb_{\al,p}$ if and only if $f\in \cb_p$ and $T_{\al}f \in \cb_p$, and
$$\norm{f}_{\cb_{\al,p}} \sim \norm{f}_{\cb_p}+\norm{T_{\al}f}_{\cb_p}.$$
\end{proposition}

\subsection{The twisted coboundary $\mathbf{\hat\delta}$}
Notice that we shall work on $n^{\text{th}}$ dimensional simplexes of $\ott$, which will be denoted by
$$
\cs_T^n=\lcl  (s_1,\ldots,s_n)\in\ott^n ; \, s_1\le s_2\le \cdots s_n\rcl.
$$
We will also set $\cac_n=\cac_n(\cs_T^n,V)$ for the continuous $n$-variables functions from $\cs_T^n$ to $V$, for a given vector space $V$. Observe that we work on those simplexes just because the operator $S_{t-u}$ is defined for $t\ge u$ (i.e. on $\cs_T^2$) only.

\smallskip

Let us recall now two basic notations of usual algebraic integration, as explained in \cite{rough} and also recalled in \cite{GT}: we define first an coboundary operator, denoted by $\der$, which acts on the set $\cac_n=\cac_n(\cs_T^n,V)$ of the continuous $n$-variables functions according to the formula:
\begin{equation}\label{defi-delta}
\der: \cac_n \to \cac_{n+1} \quad , \quad (\der g)_{t_1 \ldots t_{n+1}}=\sum_{i=1}^{n+1}(-1)^i g_{t_1 \ldots \hat{t_i} \ldots t_n}
\end{equation}
where the notation $\hat{t_i}$ means that this particular index is omitted. In this definition, $V$ stands for any vector space. Next, a convention for products of elements of $\cac_n$ is needed, and it is recalled in the following notation:
\begin{notation}\label{convention-indices}
If $g\in \cac_n(\cl(V,W))$ and $h\in \cac_m(W)$, then the product $gh \in \cac_{m+n-1}(W)$ is defined by the formula
$$(gh)_{t_1 \ldots t_{m+n-1}}=g_{t_1 \ldots t_n}h_{t_n \ldots t_{n+m-1}}.$$
\end{notation}

In point \textbf{(b)} of the introduction, we (briefly) explains why the standard increment $\der$ was not really well-suited to the study of (\ref{equation-forme-faible}). We will rather use a twisted version of $\der$, denoted by $\delha$, and defined by:
\begin{definition}
For any $n\in \N^\ast,y\in \cac_n(\cb_{\al,p})$, for all $t_1\leq \ldots \leq t_{n+1}$,
\begin{equation}
(\delha y)_{t_{n+1}\ldots t_1}=(\der y)_{t_{n+1}\ldots t_1}-a_{t_{n+1}t_n} y_{t_n\ldots t_1}, \quad \text{with} \ a_{ts}=S_{t-s}-\id \ \text{si} \ s\leq t.
\end{equation}
\end{definition} 

The operator $a:(t,s) \mapsto a_{ts}$ is only defined on the simplex $\{t \geq s \}$. As a consequence, we will have to pay attention to the decreasing order of the time variables throughout our  calculations below. Note that we will often resort to the notation $S_{ts}$ for $S_{t-s}$, so as to get a consistent notational convention for the indexes.

\smallskip

The rest of this subsection is devoted to the inventory of some of those results. The associated proofs can be found in \cite{GT}.

\smallskip 

Let us start with the fundamental property:

\begin{proposition}
The operator $\delha$ satisfies the cohomolgical relation $\delha \delha =0$. Besides, $\text{Ker} \, \delha_{|\cac_{n+1}(\cb_{\al,p})}=\text{Im} \, \delha_{|\cac_{n}(\cb_{\al,p})}$.
\end{proposition}

Now, let us turn to a more trivial result, which will be exploited in the sequel. Remember that we use the notational convention \ref{convention-indices} for time variables.  

\begin{proposition}
If $L\in \cac_{n-1}(V)$ and $M\in \cac_2(\cl(V))$, then
\begin{equation}\label{relation-alg-m-l}
\delha (M L)=(\delha M)L-M(\der L).
\end{equation}
\end{proposition}

The following result is the equivalent of Chasles relation in the $\delha$ setting. It is an obvious consequence of the multiplicative property of $S$.

\begin{proposition}
Let $x$ a differentiable process. If $y_{ts}=\int_s^t S_{tu} \, dx_u \, f_u$, then $(\delha y)_{tus}=0$ for all $s\leq u\leq t$.
\end{proposition} 

From an analytical point of view, the notion of Hölder-regularity of a process should be adapted to this context, and thus, we define, for any $\al\in [0,1/2)$, $p\in \N^\ast$, $\ka\in (0,1)$,
\begin{equation}
\cacha_1^\ka(\cb_{\al,p}):=\{ y \in \cac_1(\cb_{\al,p}): \ \sup_{s<t} \frac{\cn[(\delha y)_{ts};\cb_{\al,p}]}{\lln t-s \rrn^\ka} < \infty \}.
\end{equation} 

Let us take profit of this subsection to introduce the Hölder spaces commonly used in the $k$-increment theory. They are the subspaces of $\cac_1(V)$, $\cac_2(V)$ and $\cac_3(V)$ respectively induced by the norms ($V$ stands for any normed vector space):
$$\cn[y;\cac_1^\ka(V)]:= \sup_{s<t} \frac{\cn[(\der y)_{ts};V]}{\lln t-s \rrn^\ka},\quad \cn[y;\cac_2^\ka(V)]:=\sup_{s<t} \frac{\cn[y_{ts};V]}{\lln t-s \rrn^\ka},$$
$$\cn[y;\cac_3^{\ka,\rho}(V)]:=\sup_{s<u<t} \frac{\cn[y_{tus};V]}{\lln t-u \rrn^\ka \lln u-s \rrn^\rho},$$
$$\cn[y;\cac_3^{\mu}(V)]:=\inf \lcl \sum_i \cn[y^i;\cac_3^{\ka,\mu-\ka}]: \ y =\sum_i y^i \rcl.$$

Now, let us state the main result of this section, which allows to invert the twisted coboundary operator $\hat\der$. 

\begin{theorem}\label{existence-laha}
Fix a time $T>0$, a parameter $\ka \geq 0$ and let $\mu >1$. For any $h\in \cac_3^\mu([0,T];\cb_{\ka,p})\cap \text{Ker} \, \delha_{|\cac_3(\cb_{\ka,p})}$, there exists a unique element $$\Laha h \in \cap_{\al\in [0,\mu)} \cac_2^{\mu-\al}([0,T];\cb_{\ka+\al,p})$$ such that $\delha(\Laha h)=h$. Moreover, $\Laha h$ satisfies the following contraction property: for all $\al\in [0,\mu)$,
\begin{equation}\label{contraction-laha}
\cn[\Laha h;\cac_2^{\mu-\al}([0,T];\cb_{\ka+\al,p})] \leq c_{\al,\mu,T} \, \cn[h;\cac_3^\mu([0,T];\cb_{\ka,p})].
\end{equation}
\end{theorem}

The link between the operator $\Laha$ and a more classical formulation of the rough integration theory by means of Riemann sums, is given by the following result:

\begin{proposition}\label{lien-laha-int}
Let $g \in \cac_2(\cb_{\ka,p})$ such that $\delha g \in \cac_3^\mu(\cb_{\ka,p})$ with $\mu >1$. Then the increment $\delha f=(\id-\Laha \delha)g \in \cac_2(\cb_{\ka,p})$ satisfies
$$(\delha f)_{ts}=\lim_{|\Pi_{ts}| \to 0} \sum_{(t_k)\in \Pi_{ts}} S_{tt_{k+1}} g_{t_{k+1}t_k} \quad \text{in} \ \cb_{\ka,p},$$
for all $s\leq t$.
\end{proposition}

\section{Young case} \label{section-cas-young}

The appellation 'Young case' is used in this section order to indicate that only expansions at first order will be involved in this section. Although this kind of considerations has already been explored in \cite{glt} under more general hypotheses concerning the spatial regularity of the noise, we think that it is worthwhile to illustrate in a simple setting the adaptation of the \textit{dissection} method to the convolutional context. We will see in Section \ref{section-appli-mbf-13} that the general result of Theorem \ref{theo-cas-young} can be applied to a noise generated by a (finite-dimensional) $\ga$-Hölder process $x$, with $\ga >1/2$. This is an improvement with respect to \cite{GT}, where the unnatural condition $\ga>5/6$ had to be assumed.

\smallskip

Throughout this section, we fix a parameter $\ga \in (1/2,1)$, which (morally) represents the Hölder regularity of $x$.

\subsection{Interpretation of the integral} The aim here is to give an interpretation of the twisted Young integral $\int_s^t S_{tu} \, dx_u \, z_u$ in terms of $\delta$ and $\hat\Lambda$, and to do so, we shall follow the same reasoning as in \cite{rough,GT}: we assume first that $x$ and $z$ are smooth processes, and obtain a dissection of the integral $\int_s^t S_{tu} \, dx_u \, z_u$ in terms of $\delta$ and $\hat\Lambda$ in this particular case. This allows then to extend the notion of twisted integral to Hölder continuous signals with Hölder continuity coefficient greater than $1/2$.

\smallskip

Thus, assume, at first, that $x$ is real valued and regular (for instance lipschitz, or even differentiable) in time, as well as the integrand $z$, and look at the decomposition
\begin{equation}\label{dec-ele-inf}
\int_s^t S_{tu} \, dx_u \, z_u=\lp \int_s^t S_{tu} \, dx_u \rp z_s+\int_s^t S_{tu} \, dx_u \, (\der z)_{us}.
\end{equation}
Now, if we set $r_{ts}=\int_s^t S_{tv} \, dx_v \, (\der z)_{vs}$, one has
$$(\delha r)_{tus}=\int_s^t S_{tv} \, dx_v \, (\der z)_{vs}-\int_u^t S_{tv} \, dx_v \, (\der z)_{vu}-S_{tu} \int_s^u S_{uv} \, dx_v \, (\der z)_{vs},$$
which, using the fact that $S_{tu} S_{uv}=S_{tv}$, reduces to
\begin{equation}\label{delha-r}
(\delha r)_{tus}=\lp \int_u^t S_{tv} \, dx_v \rp (\der z)_{us}.
\end{equation}

This first elementary step lets already emerge the object which plays the role of the a priori first order increment associated to the heat equation, namely
$$X^{x,i}_{ts}=\int_s^t S_{tv} \, dx^i_v.$$

We are then in position to invert $\delha$ in (\ref{delha-r}) thanks to Theorem \ref{existence-laha}. Indeed, one easily deduces, owing to the regularity of $x$ and $z$,
$$X^x(\der z) \in \cac_3^{2}(\cb_{\al,p}) \quad \text{for some $\al \in [0,1/2)$}.$$
As a result, we get
\begin{equation}\label{int-young-inf}
\int_s^t S_{tu} \, dx_u \, z_u=X^{x,i}_{ts} z^i_s+\Laha_{ts}\lp X^{x,i}\, \der z^i  \rp.
\end{equation}

\smallskip

As in the standard case algebraic integration setting in the Young setting, we now wonder if the right-hand-side of (\ref{int-young-inf}) remains well-defined in a less regular context:
\begin{itemize}
\item \textit{From an analytical point of view}. The regularity assumption of Theorem \ref{existence-laha} imposes the condition: for all $i\in \{1,\ldots,N\}$,
$$\  X^{x,i} \der  z^i \in \cac_3^\mu(\cb_{\al,p}) \quad \text{with $\al \in [0,1/2)$ and $\mu >1$}.$$
Therefore, we shall be led to suppose that $z^i$ is $\ka$-Hölder (in the classical sense), with values in a space $\cb_{\al',p}$ to be precised, or in other words $ z^i\in \cac_1^\ka(\cb_{\al',p})$, and we will also assume that $X^{x,i} \in \cac_2^\ga(\cl(\cb_{\al',p},\cb_{\al,p}))$, with $\ka+\ga >1$. In fact, we will see that changing space is not necessary when we apply $X^{x,i}$, so that it will be sufficient to consider the case $\al=\al'$.
\item \textit{From an algebraic point of view}. We know that $\Laha$ is defined on the spaces $\cac_3^\mu(\cb_{\al,p}) \cap \text{Ker} \ \delha$. This constrains us to assume that $\delha (X^{x,i}\, \der  z^i)=0$, which, by (\ref{relation-alg-m-l}), is satisfied once we admit that $\delha X^{x,i}=0$. 
\end{itemize}

Let us record those two conditions under the abstract hypothesis:

\begin{hypothesis}\label{hypo-x-x-young}
From $x$, one can build processes $X^{x,i}$ ($i\in \lcl 1,\ldots,N\rcl$) of two variables such that, for all $i$:
\begin{itemize}
\item For any $\al \in [0,1/2)$ such that $2\al p >1$, $X^{x,i} \in \cac_2^\ga(\cl(\cb_{\al,p},\cb_{\al,p}))$
\item The algebraic relation $\delha X^{x,i}=0$ is satisfied.
\end{itemize}
\end{hypothesis}

\begin{remark}
Actually, the additional condition $2\al p >1$ could have been skipped in the latter hypothesis. We have notified it so that Hypothesis \ref{hypo-x-x-young} meets the more general Hypothesis \ref{hypo-regu} of Section \ref{section-cas-rough}.
\end{remark}

We are then allowed to use the expression (\ref{int-young-inf}) for irregular integrands:

\begin{proposition}\label{prop-int-young}
Under the assumption (\ref{hypo-x-x-young}), we define, for all processes $z$ such that $ z^i\in \cac_1^0(\cb_{\ka,p}) \cap \cac_1^\ka(\cb_p)$, $i=1,\dots,N$, with $\ka < \ga$ and $\ka+\ga >1$, the integral
\begin{equation}
\cj_{ts}(\hat{d}x \, z)=X^{x,i}_{ts} z^i_s+\Laha_{ts}\lp X^{x,i} \, \der z^i  \rp.
\end{equation}
In that case:
\begin{itemize}
\item $\cj(\hat{d}x \, z)$ is well-defined and there exists an element $\hat{z} \in \cacha_1^\ga(\cb_{\ka,p})$ such that $\delha \hat{z}$ is equal to $\cj(\hat{d}x \, z)$.
\item It holds that
\begin{equation}\label{estimation-cas-young}
\cn[\hat{z};\cacha_1^\ga(\cb_{\ka,p})] \leq c_x \lcl \cn[ z;\cac_1^0(\cb_{\ka,p})]+\cn[z;\cac_1^\ka(\cb_p)] \rcl,
\end{equation}
with
\begin{equation}
c_x \leq c \lcl \cn[X^x;\cac_2^\ga(\cl(\cb_p,\cb_p))]+\cn[X^x;\cac_2^\ga(\cl(\cb_{\ka,p},\cb_{\ka,p})] \rcl
\end{equation}
\item The integral can be written as
\begin{equation}\label{int-young-somme}
\cj_{ts}(\delha x \, z)=\lim_{|\Delta |\to 0} \sum_{(t_k) \in \Delta} S_{tt_{k+1}} X^{x,i}_{t_{k+1}t_k} z^i_{t_k},
\end{equation}
where the limit is taken over partitions $\Delta_{[s,t]}$ of the interval $[s,t]$, as their mesh tends to 0. Hence it coincides with the Young type integral $\int_s^t S_{tu} \, dx_u \, z_u$.
\end{itemize}

\end{proposition}

\begin{proof}
The fact that $\cj_{ts}(\hat{d}x \, z)$ is well defined is a direct consequence of Hypothesis \ref{hypo-x-x-young}, and the Chasles relation $\delha \cj(\hat{d}x \, z)$, which accounts for the existence of $\hat{z}$, can be shown by straightforward computations using~\eqref{relation-alg-m-l}.

\smallskip

For the second point, notice that, thanks to Hypothesis \ref{hypo-x-x-young}, one has
\begin{multline*}
\cn[\cj(\hat{d}x \, z);\cac_2^\ga(\cb_{\ka,p})]\\
\leq  \cn[X^{x,i};\cac_2^\ga(\cl(\cb_{\ka,p},\cb_{\ka,p}))] \,  \cn[ z^i;\cac_1^0(\cb_{\ka,p})]+\cn[\Laha(X^{x,i} \, \der z^i);\cac_2^\ga(\cb_{\ka,p})],
\end{multline*}
since $X^{x,i}\, \der z^i \in \cac_3^{\ga+\ka}(\cb_p)$. By the contraction property (\ref{contraction-laha}) of $\Laha$, we then deduce
$$\cn[\Laha(X^{x,i} \, \der z^i);\cac_2^\ga(\cb_{\ka,p})] \leq c \, \cn[X^{x,i};\cac_2^\ga(\cl(\cb_p,\cb_p))] \, \cn[z^i;\cac_1^\ka(\cb_p)],$$
which completes the proof of (\ref{estimation-cas-young}). According to Proposition \ref{lien-laha-int}, (\ref{int-young-somme}) is a consequence of the reformultion
\begin{equation}\label{reformulation-young}
\cj(\hat{d}x \, z)=(\text{Id}-\Laha \delha)(X^{x,i} z^i).
\end{equation}

\end{proof}

\begin{remark}
The careful readers may wonder if the starting decomposition (\ref{dec-ele-inf}) is really the most relevant choice as far as the stability of the $\delha$ structure is concerned. Indeed, at first glance, it seems more appropriate to look for an expression written in terms of $\delha$ only, and thus, for a one-dimensional noise $x$, we would rather rely on the decomposition 
\begin{equation}\label{autre-decompo-young}
\int_s^t S_{tu} \, dx_u \, z_u=S_{ts} \der x_{ts} \, z_s+\int_s^t S_{tu} \, dx_u \, (\delha z)_{us}.
\end{equation}
The order-one operator then becomes $\tilde{X}^x_{ts}=S_{ts} \der x_{ts}=\int_s^t S_{tu} \, dx_u \, S_{us}$, which coincides with the first order operator built in \cite{GT}.

\smallskip

However, one must keep in mind the particular form of the integrand we are about to consider in the system (\ref{syst-gen}), namely $z=f(y)$, for some non-linear function $f$. In order to settle a fixed-point argument, we will have to control the regularity of this integrand according to the regularity of $y$ and to do so, we can only resort to the standard tools of differential calculus, which are not consistent with the $\delha$ formalism. In other words, if one wants to estimate the norm of $\delha f(y)$, one is forced to estimate the norm of the classical increment $\der f(y)$ first. In this context, the decompositions (\ref{dec-ele-inf}) and (\ref{autre-decompo-young}) give rise to similar treatments. 

\end{remark}

\subsection{Resolution of the associated differential system} 

Using the formalism we have just introduced, we are going to show the following result of existence and uniqueness of a global solution.
To begin with, let us state the assumption on the regularity of the functions $\sigma_i$ appearing in the definition of the vectorfields $f_i$, $i=1,\dots,N$.

\begin{hypothesis}
\label{def:reg-f}
Let $f:\cB_p \to \cB_p$ be a vector field defined by $f(\varphi)(\xi) = \sigma(\xi,\varphi(\xi))$ for some function $\sigma: \R^n\times \R \to \R$.
We say that $f \in \mathcal{X}_k$ for $k\ge 1$ if $\sigma$ is of uniformly compact support in the first variable, in the following sense: $\si_i:\R^n\times\R\to\R$ is such that $\si_i(\cdot,\eta)=0$ outside of a ball $B_{\R^N}(0,M)$, independently of $\eta\in\R$.

\smallskip

In order to be element of $\mathcal{X}_k$, we also ask to a vector field $f$ to satisfy the following inequality:
$$
\sup_{\xi\in\R^n,\eta\in \R} \max_{n=0,\dots,k}|\nabla_\eta^n \sigma_i(\xi,\eta)|+\max_{n=0,\dots,k-1}|\nabla_\xi \nabla_\eta^n \sigma_i(\xi,\eta)|
 < +\infty .
$$
\end{hypothesis}

A direct application of Proposition \ref{result-stri} easily leads to:
\begin{proposition}\label{prop-stri}
If $f \in \mathcal{X}_1$, then for any $\vp \in \cb_{\al,p}$, $f(\vp) \in \cb_{\al,p}$ and
$$\cn[f(\vp);\cb_{\al,p}] \leq c_f \lcl 1+\cn[\vp;\cb_{\ka,p}] \rcl.$$
\end{proposition}

The following notation will also be used in the sequel of the paper.
\begin{notation}\label{not:ineq-c-a}
Let $A,B$ be two positive quantities, and $a$ a parameter lying in a certain vector space $V$. We say that $A\lesssim_{a}B$ if there exists a positive constant $c_a$ depending on $a$ such that $A\le c_a B$. When we don't want to specify the dependence on $a$, we just write $A\lesssim B$. Notice also that the value of the constants $c$ or $c_a$ in our computations can change from line to line, throughout the paper.
\end{notation}

\smallskip

We are now ready to prove the main theorem of this section:
\begin{theorem}\label{theo-cas-young}
Assume  Hypothesis~\ref{hypo-x-x-young} with $\ga>1/2$, and assume also that $f=(f_1,\ldots,f_N)$ with $f_i\in \cx_2$ for $i=1,\dots,N$. For any $\ka < \ga$ such that $\ga+\ka >1$ and $2\ka p >n$, consider the space $\cacha_1^{0,\ka}([0,T],\cb_{\ka,p})=\cac_1^0([0,T],\cb_{\ka,p}) \cap \cacha_1^\ka([0,T],\cb_{\ka,p})$, provided with the norm
$$
\cn[ .; \cacha_1^{0,\ka}([0,T],\cb_{\ka,p})]=\cn[.;\cac_1^0([0,T],\cb_{\ka,p})]+\cn[.;\cacha_1^\ka([0,T],\cb_{\ka,p})].
$$
Then the infinite-dimensional system
\begin{equation}\label{syst-young}
(\delha y)_{ts}=\cj_{ts}(\hat{d}x \, f(y)) \quad , \quad y_0=\psi \in \cb_{\ka,p},
\end{equation} 
interpreted with Proposition \ref{prop-int-young}, admits a unique global solution in $\cacha_1^{0,\ka}([0,T],\cb_{\ka,p})$. Besides, the Itô application $(\psi,X^{x,i}) \mapsto y$, where $y$ is the unique solution of (\ref{syst-young}), is Lipschitz.
\end{theorem}

\begin{remark}
In the last statement, we consider the operators $X^{x,i}$  as elements of the incremental space $\cac_2^\ga(\cl(\cb_{p},\cb_p)) \cap \cac_2^\ga(\cl(\cb_{\ka,p},\cb_{\ka,p}))$. The regularity of the Itô application with respect to $X^{x,i}$ is then relative to the norm
$$\cn[.;\cac \cl^{\ka,\ga,p}]=\cn[.;\cac_2^\ga(\cl(\cb_p,\cb_p))]+\cn[.;\cac_2^\ga(\cl(\cb_{\ka,p},\cb_{\ka,p}))].$$
\end{remark}

\begin{proof}
It is a classical fixed-point argument. We will only prove the existence and uniqueness of a local solution. The reasoning which enables to extend the local solution into a solution on the whole interval $[0,T]$ is standard; some details about the general procedure can be found in \cite{rough} (in a slightly different context). 

\smallskip

We consider an interval $I=[0,T_\ast]$ with $T_\ast$ a time that may change during the proof, and the application $\Gamma:\cacha_{1,\psi}^{0,\ka}(I, \cb_{\ka,p}) \to \cacha_{1,\psi}^{0,\ka}(I,\cb_{\ka,p})$ defined by $\Gamma(y)_0=\psi$ and $(\delha \Gamma(y))_{ts}=\cj_{ts}(\hat{d}x \, f(y))$.

\smallskip

\noindent
\textit{Invariance of a ball}. Let $y\in \cacha_{1,\psi}^{0,\ka}(I, \cb_{\ka,p})$ and $z=\Gamma(y)$. By (\ref{estimation-cas-young}), we know that
\begin{equation}\label{prem-est-young}
\cn[z;\cacha_1^\ka(I,\cb_{\ka,p})]
\leq c_x \lln I\rrn^{\ga-\ka} \lcl \cn[f_i(y);\cac_1^\ka(I,\cb_p)]+\cn[f_i(y);\cac_1^0(\cb_{\ka,p})] \rcl.
\end{equation}
Recalling our convention in Notation \ref{not:ineq-c-a}, the assumption $f_i\in \cx_1$ is enough  to guarantee that the following bounds holds for $f_i$:
$\cn[f_i(\varphi)-f_i(\psi);\cb_{p}] \lesssim_f \cn[\varphi-\psi;\cb_{p}]$
and $\cn[f_i(\varphi);\cb_{\ka,p}] \lesssim_f 1+\cn[\varphi;\cb_{\ka,p}]
$ for arbitrary test functions $\varphi,\psi$.
So we have 
\bean
\cn[f_i(y);\cac_1^\ka(I,\cb_p)] &\lesssim_f & \cn[y;\cac_1^\ka(I,\cb_p)]\\
&\lesssim_f &  \cn[y;\cac_1^0(I,\cb_{\ka,p})]+\cn[y;\cacha_1^\ka(I,\cb_{\ka,p})]  \\
& \lesssim_f &  \cn[y;\cacha_1^{0,\ka}(I,\cb_{\ka,p})],
\eean
where, to get the second inequality, we have used the property (\ref{regu-hold-semi}) of the semigroup. We get also
$\cn[f_i(y);\cac_1^0(\cb_{\ka,p})] \lesssim_f 1+\cn[y;\cac_1^0(\cb_{\ka,p})] $, which, going back to (\ref{prem-est-young}), leads to
$$\cn[z;\cacha_1^\ka(I,\cb_{\ka,p})] \lesssim_{x,f} \lln I\rrn^{\ga-\ka} \lcl 1+\cn[y;\cacha_1^{0,\ka}(I,\cb_{\ka,p})] \rcl.$$
Besides, $z_s=(\delha z)_{s0}+S_s \psi$, hence, since $S_s$ is a contraction operator on $\cb_{\ka,p}$,
$$\cn[z;\cac_1^0(I,\cb_{\ka,p})] \leq \lln I\rrn^\ka \cn[z;\cacha_1^\ka(I,\cb_{\ka,p})]+\norm{\psi}_{\cb_{\ka,p}}.$$
Finally,
$$\cn[z;\cacha_1^{0,\ka}(I,\cb_{\ka,p})] \leq \norm{\psi}_{\cb_{\ka,p}}+c_x \lln I \rrn^{\ga-\ka}\lcl 1+ \cn[y;\cacha_1^{0,\ka}(I,\cb_{\ka,p})] \rcl.$$
Then we choose $I=[0,T_1]$ such that $c_x T_1^{\ga-\ka} \leq \frac{1}{2}$ to get the invariance by $\Gamma$ of the balls
$$B_{T_0,\psi}^R=\{ y \in \cacha_1^{0,\ka}([0,T_0],\cb_{\ka,p}): \ y_0=\psi, \quad \cn[y;\cac_{1}^{0,\ka}([0,T_0],\cb_{\ka,p})] \leq R  \},$$
for any $T_0 \leq T_1$, with (for instance) $R=1+2 \norm{\psi}_{\cb_{\ka,p}}$.

\smallskip

\noindent
\textit{Contraction property}. Let $y,\yti \in \cacha_{1,\psi}^{0,\ka}(I,\cb_{\ka,p})$ and $z=\Gamma(y),\zti=\Gamma(\yti)$. By (\ref{estimation-cas-young}), 
\begin{multline}\label{decompo-contr}
\cn[z-\zti;\cacha_1^\ka(\cb_{\ka,p})] \leq \\
c_x \lln I \rrn^{\ga-\ka} \lcl \cn[ f_i(y)-f_i(\yti);\cac_1^0(\cb_{\ka,p})]+\cn[f_i(y)-f_i(\yti); \cac_1^\ka(\cb_p)] \rcl.
\end{multline}
In order to estimate the Hölder norm $\cn[f_i(y)-f_i(\yti);\cac_1^\ka(\cb_p)]$, we rely on the decomposition
\begin{multline*}
\sigma_i(\xi,y_t(\xi))-\sigma_i(\xi,\yti_t(\xi))-\sigma_i(\xi,y_s(\xi))+\sigma_i(\xi,\yti_s(\xi))\\
=\der(y-\yti)_{ts}(\xi) \int_0^1 dr \, \sigma_i'(\xi,y_s(\xi)+r(\der y)_{ts}(\xi))\\
+(\der \yti)_{ts}(\xi) \int_0^1 dr \lcl \sigma_i'(\xi,y_s(\xi)+r(\der y)_{ts}(\xi))-\sigma_i'(\xi,\yti_s(\xi)+r(\der \yti)_{ts}(\xi)) \rcl.
\end{multline*}
Therefore,
\begin{multline*}
\cn[ f_i(y)-f_i(\yti);\cac_1^\ka(\cb_p)]\\
\leq c_f \lcl \cn[y-\yti;\cacha_1^{0,\ka}(\cb_{\ka,p})]+\cn[\yti;\cacha_1^{0,\ka}(\cb_{\ka,p})] \, \cn[y-\yti;\cac_1^0(\cb_\infty)] \rcl.
\end{multline*}
Remember that we have assumed that $2\ka p >n$, so that, by the Sobolev continuous inclusion $\cb_{\ka,p} \subset \cb_\infty$ , $\cn[y-\yti;\cac_1^0(\cb_\infty)] \leq \cn[y-\yti;\cac_1^0(\cb_{\ka,p})]$ and as a result
$$\cn[f_i(y)-f_i(\yti);\cac_1^\ka(\cb_p)] \leq c \, \cn[y-\yti;\cacha_1^{0,\ka}(\cb_{\ka,p})] \lcl 1+\cn[\yti;\cacha_1^{0,\ka}(\cb_{\ka,p})]\rcl.$$
The same kind of argument easily leads to
\begin{multline*}
\cn[f_i(y)-f_i(\yti);\cac_1^0(\cb_{\ka,p})]\\
\leq c \, \cn[y-\yti;\cacha_1^{0,\ka}(\cb_{\ka,p})] \lcl 1+\cn[y;\cacha_1^{0,\ka}(\cb_{\ka,p})]+\cn[\yti;\cacha_1^{0,\ka}(\cb_{\ka,p})]\rcl,
\end{multline*}
The last two estimations, together with (\ref{decompo-contr}), provide a control of $\cn[z-\zti;\cacha_1^\ka(\cb_{\ka,p})]$ in terms of $y,\yti$.
Moreover, as $z_0=\zti_0=\psi$, 
$$\cn[z-\zti;\cac_1^0(\cb_{\ka,p})] \leq \lln I \rrn^\ka \cn[z-\zti;\cacha_1^\ka(\cb_{\ka,p})].$$
Now, if $y,\yti$ both belong to one of the invariant balls $B_{T_0,\psi}^R$, with $T_0 \leq T_1$, the above results give
$$\cn[z-\zti;\cacha_1^{0,\ka}([0,T_0],\cb_{\ka,p})] \leq c_x T_0^{\ga-\ka} \lcl 1+2R \rcl \cn[y-\yti;\cacha_1^{0,\ka}([0,T_0],\cb_{\ka,p})].$$
It only remains to pick $T_0 \leq T_1$ such that $c_x T_0^{\ga-\ka} \lcl 1+2R\rcl \leq \frac{1}{2}$, and we get the contraction property of the application $\Gamma: B_{T_0,\psi}^R \to B_{T_0,\psi}^R$. This statement obviously completes the proof of the existence and uniqueness of a solution to (\ref{syst-young}) defined on $[0,T_0]$.

\end{proof}

\section{Rough case}\label{section-cas-rough}

The aim of this section is to go one step further in the rough path procedure: We would like to conceive more sophisticated developments of the integral so as to cope with a $\ga$-Hölder driving process, with $\ga \in (1/3,1/2)$.

\subsection{Heuristic considerations}\label{section-rough-heuri}
The strategy to give a (reasonable) sense to the integral $\int_s^t S_{tu} \, dx^i_u \, f_i(y_u)$ will be largely inspired by the reasoning followed for the standard integral $\int_s^t y_u \, dx_u$, explained in \cite{rough,GT}. Thus, let us suppose at first that the process $x$ is differentiable (in time), as a function with values in a Banach space. The procedure to reach a suitable decomposition of the integral divides into two steps:
\begin{itemize}
\item Identify the space $\cq$ of controlled processes which will accomodate the solution of the system.
\item Decompose $\int_s^t S_{tu} \, dx^i_u \, f_i(y_u)$ as an element of $\cq$ when $y$ belongs itself to $\cq$, until we get an expression likely to remain meaningful if $x$ is less regular.
\end{itemize}

This heuristic reasoning essentially aims at identifying the algebraic structures which will come into play. The details concerning the analytical conditions will be checked a posteriori. The noisy nonlinearity is given by equation (\ref{forme-bruit}), namely 
\begin{equation*}
X_t(\varphi)=\sum_{i=1}^N  \, x^{i}_t f_i(\varphi),
\quad\mbox{with}\quad
f_i(\varphi)(\xi) = \sigma_i(\xi,\varphi(\xi)),
\end{equation*}
and we shall see that $\si_i$ has to be considered as an element of $\cx_2$, as defined in Hypothesis~\ref{def:reg-f}.

\smallskip

\noindent
\textit{Step 1: Identification of the controlled processes}. The first elementary decomposition still consists in:
\begin{equation}\label{proc-contr-1}
\int_s^t S_{tu} \, dx^i_u \, f_i (y_u)=\lp \int_s^t S_{tu} \, dx^i_u \rp f_i(y_s)+\int_s^t S_{tu} \, dx^i_u \, \der(f_i(y))_{us}.
\end{equation}

It is then natural to think that the potential solution of the system is to belong to a space structured by the relation
$$(\delha y)_{ts}=\lp \int_s^t S_{tu} \, dx^i_u \rp y^{x,i}_s+y^\sharp_{ts},$$
with $y^\sharp$ admitting a Hölder regularity twice higher than $y$. For the solution itself, we would have $y^{x,i}_s=f_i(y_s)$, $y^\sharp_{ts}=\int_s^t S_{tu} \, dx^i_u \, \der(f_i(y))_{us}$ hence the potential algebraic structure of the controlled processes
$$\cq=\{ y: \ \delha y=X^{x,i}_{ts} y^{x,i}_s+y^\sharp_{ts}\}, \ \text{with} \ X^{x,i}_{ts}=\int_s^t S_{tu} \, dx_u^i.$$
Remember that the latter operator satisfies the algebraic relation
\begin{equation}\label{rel-alg-xx}
\delha X^{x,i}=0.
\end{equation}
Besides, it will turn out useful in the sequel to write $X^{x,i}$ as
\begin{equation}\label{decompo-x-x-1}
X^{x,i}=X^{ax,i}+\der x^i \quad , \quad \text{with} \ X^{ax,i}_{ts}=\int_s^t a_{tv} \, dx^i_v.
\end{equation}
Morally, $X^{ax,i}$ admits a higher Hölder regularity than $x$ owing to the property (\ref{regu-hold-semi}) of the semigroup. We will go back over the usefulness of this trivial decomposition in Remark~\ref{rk-x-x-x}.
In the following we will omit sometimes the vector indexes $i,j,\dots$ whenever the contractions are obvious. 

\smallskip

\noindent
\textit{Step 2: Decomposition of $\int_s^t S_{tu} \, dx_u \, f_i(y_u)$ when $y\in \cq$}. Going back to expression~(\ref{proc-contr-1}), we see that it is more exactly the integral $\int_s^t S_{tu} \, dx_u \, \der(f_i(y))_{us}$ that remains to be dissected when $y\in \cq$, that is to say when the $\delha$-increment of $y$ can be written as
$(\delha y)_{ts}=X^{x,i}_{ts}y^{x,i}_s+y^\sharp_{ts}$. 
To this purpose, let us introduce a new notation which will appear in many of our future computations:
\begin{notation}\label{not:f-prime}
For any $f\in\cx_2$ as defined in Hypothesis \ref{def:reg-f}, we set
$$
[f'(\varphi)](\xi) = \nabla_2 \sigma(\xi,\varphi(\xi)),
$$
where $\nabla_2$ stands for the derivative with respect to the second variable. The function $f'$ is understood as a mapping from $\cb_p$ to $\cb_p$ for any $p\ge 1$.
\end{notation}

Using this notational convention, notice that
\begin{equation}
\label{decompo-f-un}
\begin{split}
\der(f_i(y))_{ts} &= (\der y)_{ts} \cdot f'_i(y_s)+\int_0^1 dr \, \lc f'_i(y_s+r(\der y)_{ts})-f_i'(y_s) \rc \cdot (\der y)_{ts}\\
&= (a_{ts}y_s) \cdot f'_i(y_s)+(\delha y)_{ts} \cdot f'_i(y_s)+f_i(y)^{\sharp,1}_{ts}\\
&= (a_{ts}y_s) \cdot f_i'(y_s)+(X^{x,j}_{ts}y^{x,j}_s)\cdot f_i'(y_s)+f_i(y)^{\sharp,1}_{ts}+f_i(y)^{\sharp,2}_{ts}\\
&= (a_{ts}y_s) \cdot f_i'(y_s)+(\der x^j)_{ts} \cdot y^{x,j}_s \cdot f_i'(y_s)+f_i(y)^{\sharp,1}_{ts}+f_i(y)^{\sharp,2}_{ts}+f_i(y)^{\sharp,3}_{ts},
\end{split}\end{equation}
where we have successively introduced the notations
\begin{equation}\label{def-exa-f-et-1}
f_i(y)^{\sharp,1}_{ts}=\int_0^1 dr \, \lc f_i'(y_s+r(\der y)_{ts})-f_i'(y_s) \rc \cdot (\der y)_{ts} \quad , \quad f_i(y)^{\sharp,2}_{ts}=y^\sharp_{ts} \cdot f_i'(y_s),
\end{equation}
\begin{equation}\label{def-exa-f-et-2}
f_i(y)^{\sharp,3}_{ts}=(X^{ax,j}_{ts}y^{x,j}_s)\cdot f_i'(y_s).
\end{equation}

Observe that, in the course of those computations, we have used some additional conventions that we make explicit for further use:
\begin{notation}\label{not:bilinear-B}
Let $\vp,\psi$ be two elements of $\cb_p$. Then $\vp\cdot\psi$ is the element of $\cb_{p/2}$ defined by the pointwise multiplication $[\vp\cdot\psi](\xi)=\vp(\xi)\, \psi(\xi)$. If we assume furthermore that $M,N$ are two elements of $\cl(\cb_p;\cb_p)$, then the bilinear form $B(M\otimes N)$ is defined as:
$$
B(M\otimes N): \cb_p\times\cb_p\to\cb_{p/2}, \quad
(\vp,\psi)\mapsto [B(M\otimes N)](\vp,\psi)=M(\vp) \cdot N(\psi).
$$
\end{notation}
With this convention in mind, the algebraic decomposition (\ref{decompo-f-un}) of $f_i(y)$ can now be read as:
\begin{equation}\label{decompo-f-deux}
\der(f_i(y))_{ts}=B (a_{ts} \otimes \text{Id})(y,f'_i(y))_s+(\der x^j)_{ts} \cdot y^{x,j}_s \cdot f_i'(y_s)+f_i(y)^\sharp_{ts}.
\end{equation}

\smallskip

If we analyze the regularity of the terms of this expression, it seems reasonable to consider the first two terms as elements of order one and $f_i(y)^\sharp$ as an element of order two. Let us make two comments about this intuition:
\begin{itemize}
\item[\textbf{(a)}] To  assert that $B(a_{ts} \otimes \text{Id})(y,f_i'(y))_s$ admits a strictly positive Hölder regularity, otherwise stated to retrieve increments $\lln t-s \rrn^\al$ from the operator $a_{ts}$, we must use the property (\ref{regu-hold-semi}) of the semigroup. It means in particular that a change of space will occur: if $y_s \in \cb_{\al,p}$, then $B(a_{ts} \otimes \text{Id})(y,f_i'(y))_s$ will be estimated as an element of $\cb_p$. This remark also holds for $f_i(y)^{\sharp,3}_{ts}=(X^{ax,j}_{ts}y^{x,j}_s)\cdot f_i'(y_s)$.
\item[\textbf{(b)}] The term $f_i(y)^{\sharp,1}$ is considered as a second order element insofar as it is easily (pointwise) estimated by (a constant times) $|(\der y)_{ts}|^2$. However, as far as the spatial regularity is concerned, this supposes that $f_i(y)^{\sharp,1}$ has to be seen as an element of $\cb_{p/2}$, if $y\in \cb_p$. To go back to the base space $\cb_p$, we shall use the regularization properties (\ref{p/k-p}) of the semigroup, through the operator $X^{x}$ (Hypothesis (\ref{regu-x-x-rough})).
\end{itemize}

\smallskip

Now, inject the decomposition (\ref{decompo-f-deux}) into (\ref{proc-contr-1}) to obtain
\begin{multline}\label{decompo-int-un}
\int_s^t S_{tu} \, dx^i_u \, f_i(y_u)=X^{x,i}_{ts} f_i(y)_s+X^{xa,i}_{ts}(y, f'_i(y))_s+X^{xx,ij}_{ts}(y^{x,j}\cdot f'_i(y))_s\\
+\int_s^t S_{tu} \, dx^i_u \, f_i(y)^\sharp_{us},
\end{multline}
where we have introduced the following operators of order two (which act on some spaces which will be detailed later on):
\begin{equation}\label{heur-x-x-x}
X^{xa,i}_{ts}=\int_s^t S_{tu} \, dx_u^i \, B(a_{us} \otimes \text{Id}) \quad \text{and} \quad X^{xx,ij}_{ts}=\int_s^t S_{tu} \, dx_u^i \, (\der x^j)_{us}.
\end{equation}
A little more specifically, those operators act on couples $(\vp,\psi)$ in some Sobolev type spaces, and
\begin{equation*}
X^{xa,i}_{ts}(\vp,\psi)=\int_s^t S_{tu} \, dx_u^i \, \lc a_{us}(\vp) \cdot \psi\rc 
\quad \text{and} \quad 
X^{xx,ij}_{ts}(\vp)=\int_s^t S_{tu} \, dx_u^i \, (\der x^j)_{us} \lc \vp \rc.
\end{equation*}
Then, since we have assumed that $f_i(y)^\sharp$ admitted a "double" regularity, we can see the residual term $r_{ts}=\int_s^t S_{tu} \, dx^i_u \, f_i(y_u)^\sharp$ as a third order element, whose regularity is expected to be greater than 1 as soon as the Hölder regularity of $x$ is greater than $1/3$. Thus, we are in the same position as in (\ref{dec-ele-inf}), and just as in the latter situation, $r$ will be interpreted thanks to $\Laha$.

\smallskip

In order to compute $\delha r$, rewrite $r$ using (\ref{decompo-int-un}):
$$r_{ts}=\int_s^t S_{tu} \, dx^i_u \, f_i(y_u)-X^{x,i}_{ts}f_i(y_s)-X^{xa,i}_{ts}(y,f'_i(y))_s-X^{xx,ij}_{ts}(y^{x,j}\cdot  f'_i(y))_s.$$
Therefore, with the help of the algebraic formula (\ref{relation-alg-m-l}), we get
\begin{multline*}
(\delha r)_{tus}=X^{x,i}_{tu}\der(f_i(y))_{us}-(\delha X^{xa,i})_{tus}(y, f_i'(y))_s+X^{xa,i}_{tu} \der (y, f'_i(y))_{us}\\
-(\delha X^{xx,ij})_{tus}(y^{x,j}\cdot  f'_i(y))_s+X^{xx,ij}_{tu} \der (y^{x,j}\cdot  f'_i(y))_{us}.
\end{multline*}
Going back to the very definition of $X^{xa,i}$ and $X^{xx,ij}$, it is quite easy to show that the following relations are satisfied whenever $x$ is a smooth function:
\begin{equation}\label{rel-alg-xxa}
(\delha X^{xa,i})_{tus}=X^{xa,i}_{tu}(a_{us} \otimes \text{Id})+X^{x,i}_{tu}(a_{us} \otimes \text{Id}),
\end{equation}
\begin{equation}\label{rel-alg-xxx}
(\delha X^{xx,ij})_{tus}=X^{x,i}_{tu}(\der x^j)_{us} .
\end{equation}
By combining these two relations together with (\ref{decompo-f-deux}), we deduce
\begin{multline}\label{delha-r-un}
(\delha r)_{tus}=X^{x,i}_{tu}(f_i(y)^\sharp_{us})+X^{xa,i}_{tu}((\delha y)_{us} ,f'_i(y_s))+X^{xa,i}_{tu}(y_u ,  \der(f_i'(y))_{us})\\
+X^{xx,ij}_{ts} \der(y^{x,j}\cdot f'_i(y))_{us}:=J_{tus}.
\end{multline}

All the terms of this decomposition are (morally) of order three. Now, remember that we wish to tackle the case $3\ga >1$, so that it seems actually wise to invert $\delha$ at this point, and we get
\begin{equation}\label{int-prov}
\int_s^t S_{tu} \, dx^i_u \, f_i(y_u)=X^{x,i}_{ts} f_i(y_s)+X^{xa,i}_{ts}(y, f'_i(y))_s+X^{xx,ij}_{ts}(y^{x,j}\cdot f_i'(y))_s+\Laha_{ts}(J),
\end{equation}
where $J_{tus}$ is given by (\ref{delha-r-un}). Notice once again that we have obtained a decomposition valid for some smooth functions $x$ and $y$, but this decomposition can now be extended to an irregular situation up to $\ga>1/3$.

\smallskip

In a natural way, we will use (\ref{int-prov}) as the definition of the integral in the prescribed context of a $\ga$-Hölder process with $\ga >1/3$. To conclude this heuristic reasoning, let us summarize the different hypotheses we have (roughly) raised during the procedure:
\begin{itemize}
\item The process $x$ generates four operators $X^{x}$, $X^{ax}$, $X^{xa}$ et $X^{xx}$, which satisfy the algebraic relations (\ref{rel-alg-xx}), (\ref{rel-alg-xxa}) and (\ref{rel-alg-xxx}). As for the Hölder regularity of those operators, $X^{x}$ admits the same regularity as $x$, $X^{xx}$ twice the regularity of $x$, just as $X^{ax}$ and $X^{xa}$ (even if one must change the space one works with, according to the above point \textbf{(a)}).
\item The increments $(\delha y)_{ts}$ can be decomposed as $(\delha y)_{ts}=X^{x}_{ts}y^x_s+y^\sharp_{ts}$, where $y^\sharp$ is twice more regular than $y$. Besides, according to \textbf{(a)} again, the process $y$ must evolve in a space $\cb_{\al,p}$, with $\al >0$. These remarks will give birth to the spaces $\cq_{\al,p}^\ka$.
\item The functions $\sigma_i$ are regular enough (to be precised below).
\end{itemize}

\begin{remark}\label{rk-x-x-x}
If one has a look at the constructions established in \cite{GT}, it seems more natural, at first sight, to search for a decomposition of the integral based on the (twisted) iterated integral
\begin{equation}\label{x-x-x-tilde}
\tilde{X}^{xx,ij}_{ts}:=\int_s^t S_{tu} \, dx^i_u \,B( X^{x,j}_{us} \otimes \mathrm{Id}) =\int_s^t S_{tu} \, dx^i_u \, B \lp \int_s^u S_{uv} \, dx^j_v \otimes \mathrm{Id}\rp ,
\end{equation}
rather than on the area $X^{xx}_{ts}$ we have introduced in (\ref{heur-x-x-x}). In a way, the definition of $\tilde{X}^{xx}_{ts}$ is actually more consistent with the general iteration scheme of the rough path procedure. Nevertheless, when it comes to applying the results to a fBm $x$ (with Hurst index $H\in (1/3,1/2)$) for instance, it seems difficult to justify the existence of the iterated integral (\ref{x-x-x-tilde}). According to our computations, this difficulty is due to a lack of regularity for the term $S_{uv}$ in (\ref{x-x-x-tilde}). Indeed, if one refers to \cite{BJ}, the definition of the integral would require a condition like
$$\cn[S_{uv}-S_{uu};\cl(\cb_{\al,p},\cb_{\al,p})] \lesssim \lln u-v \rrn^\nu,$$
for some $\nu >0$, but this kind of inequality cannot be satisfied in this general form, since the Hölder property (\ref{regu-hold-semi}) of the semigroup requires a change of space. This is why we have turned to a formulation with $X^{xx}_{ts}$, which is made possible by the introduction of the operator $X^{ax}_{ts}$ (defined by (\ref{decompo-x-x-1})) in the decomposition (\ref{decompo-f-un}). As we shall see in Section~\ref{section-appli-mbf-13}, the definition and the estimation of the regularity of $X^{xx}$ are much simpler, since this can be done by means of an integration by parts argument.

\end{remark}

\subsection{Definition of the integral}

In this subsection, we will only make the previous assumptions and constructions more formal. From now on, we fix a coefficient $\ga >1/3$, which (morally) represents the Hölder regularity of the driving process $x$. The definition of the rough path above $x$ associated to the heat equation is then the following:
\begin{hypothesis}\label{hypo-regu}
We assume that the process $x$ allows to define operators $X^{x,i}$, $X^{ax,i}$, $X^{xa,i}$, $X^{xx,ij}$ ($i,j \in \{1,\ldots, N \}$), such that, recalling our Notation \ref{not:bilinear-B}:
\begin{itemize}
\item[(H1)] From an algebraic point of view:
\begin{equation}\label{hypo-alg-rough-x-x}
\delha X^{x,i}=0
\end{equation}
\begin{equation}\label{hypo-alg-rough-x-ax}
X^{x,i}=X^{ax,i}+\der x^i
\end{equation}
\begin{equation}\label{hypo-alg-rough-x-xa}
\delha X^{xa,i}=X^{xa,i}(a \otimes \text{Id})+X^{x,i}(a \otimes \text{Id})
\end{equation}
\begin{equation}\label{hypo-alg-rough-x-xx}
\delha X^{xx,ij}=X^{x,i}(\der x^j).
\end{equation}
\item[(H2)] From an analytical point of view: if $2\al p >n$, then
\begin{equation}\label{regu-x-x-rough}
X^{x,i} \in \cac_2^\ga(\cl(\cb_p,\cb_p))  \cap \cac_2^\ga(\cl(\cb_{\al,p},\cb_{\al,p}))\cap \cac_2^{\ga-n/(2p)}(\cl(\cb_{p/2},\cb_p)) 
\end{equation}
\begin{equation}\label{regu-x-a-x}
X^{ax,i}\in \cac_2^{\ga+\al}(\cl(\cb_{\al,p},\cb_p))
\end{equation}
\begin{equation}\label{regu-x-x-a}
X^{xa,i}\in \cac_2^{\ga+\al-n/(2p)}(\cl(\cb_{\al,p} \times \cb_p,\cb_p)) \cap \cac_2^\ga(\cl(\cb_{\al,p} \times \cb_{\al,p},\cb_{\al,p})) 
\end{equation}
\begin{equation}\label{regu-x-x-x}
X^{xx,ij} \in \cac_2^{2\ga}(\cl(\cb_p ,\cb_p)) \cap \cac_2^{2\ga}(\cl(\cb_{\al,p} ,\cb_{\al,p})) \cap \cac_2^{2\ga}(\cl(\cb_{\al,p} ,\cb_p)).
\end{equation}
\end{itemize}

We will denote by $\textbf{X}=(X^x,X^{ax},X^{xa},X^{xx})$ the path so defined. $\textbf{X}$ belongs to a product of operators spaces, denoted by $\cac \cl^{\ga,\ka,p}$, and furnished with a natural norm build with the norms of each space.
\end{hypothesis}

The formal definition of controlled process takes the following form:
\begin{definition}\label{defi-contr-proc}
For all $\al \in (0,1/2)$, $\ka \in (0,1)$, we define
\begin{multline*}
\hat{\cq}_{\al,p}^\ka=\hat{\cq}_{\al,p}^\ka([0,T])=\{ y \in \cacha_1^\ka([0,T],\cb_{\al,p}): \, (\delha y)_{ts}=X^{x,i}_{ts} y^{x,i}_s+y^\sharp_{ts}, \\
y^{x,i} \in \cac_1^0([0,T],\cb_{\al,p}) \cap \cac_1^\ka([0,T],\cb_p), \, y^\sharp \in \cac_2^\ga([0,T],\cb_{\al,p})\cap \cac_2^{2\ka}([0,T],\cb_p) \}.
\end{multline*}
We will call $\hat{\cq}_{\al,p}^\ka$ the space of $\ka$-controlled processes of $\cb_{\al,p}$, together with the norm
\begin{multline*}
\cn[y;\hat{\cq}_{\al,p}^\ka ] =\cn[y;\cacha_1^\ka(\cb_{\al,p})]+\sum_{i=1}^N \lcl \cn[y^{x,i};\cac_1^0(\cb_{\al,p})]+\cn[y^{x,i};\cac_1^\ka(\cb_p)]\rcl\\
+\cn[y^\sharp;\cac_2^\ga(\cb_{\al,p})]+\cn[y^\sharp;\cac_2^{2\ka}(\cb_p)],
\end{multline*}
where the time interval $[0,T]$ is omitted for sake of clarity.
\end{definition}
Observe that, in what follows, we will only consider the spaces $\hat{\cq}_{\ka,p}^\ka$, with $2\ka p> 1$.

\smallskip

We can now show how nonlinearities of the form given in Hypothesis \ref{def:reg-f} act on a controlled process.
\begin{lemma}\label{lem-dec-f} Assume that $f_i \in \cx_2$ for $i=1,\dots,N$ and let $\ka \in (1/3,\ga)$. If $y\in \hat{\cq}_{\ka,p}^\ka$ admits the decomposition $\delha y=X^{x,i}y^{x,i}+y^\sharp$, then the increment $\der f_i(y)$ can be written as
\begin{equation}\label{decomp-si-i-f}
\der(f_i(y))_{ts}=(a_{ts} \otimes \text{Id})(y,f'_i(y))_s+(\der x^j)_{ts}\cdot (y^{x,j} \cdot f'_i(y))_s+f_i(y)^\sharp_{ts},
\end{equation}
with $f_i(y)^\sharp=f_i(y)^{\sharp,1}+f_i(y)^{\sharp,2}+f_i(y)^{\sharp,3}$, where the elements $f_i(y)^{\sharp,k}$ are given by (\ref{def-exa-f-et-1}) and (\ref{def-exa-f-et-2}). Moreover, one has
\begin{equation}\label{estim-f-sha-1}
\cn[ f_i(y)^{\sharp,1};\cac_2^{2\ka}(\cb_{p/2})]\leq \cfx \lcl \cn[y;\cac_1^0(\cb_{\al,p})]^2+\cn[y;\hat{\cq}_{\ka,p}^\ka]^2 \rcl
\end{equation}
\begin{equation}\label{estim-f-sha-2-3}
\cn[ f_i(y)^{\sharp,2};\cac_2^{2\ka}(\cb_p)] \leq \cfx \, \cn[y;\hat{\cq}_{\ka,p}^\ka]  \quad , \quad \cn[ f_i(y)^{\sharp,3};\cac_2^{2\ka}(\cb_p)] \leq \cfx \, \cn[y;\hat{\cq}_{\ka,p}^\ka].
\end{equation}
\end{lemma}
\begin{proof}
This refers to the decomposition (\ref{decompo-f-un}). The estimate of $f_i(y)^{\sharp,2}$ is obvious, while the estimate of $f_i(y)^{\sharp,3}$ stems from the hypothesis (\ref{regu-x-a-x}).
As for $f_i(y)^{\sharp,1}$, notice that
$$\norm{f_i(y)_{ts}^{\sharp,1}}_{\cb_{p/2}} \lesssim c_f \norm{(\der y)_{ts}^2 }_{\cb_{p/2}}\lesssim \norm{(\der y)_{ts} }_{\cb_{p}}^2\lesssim \norm{(\delha y)_{ts} }_{\cb_{p}}^2+\norm{a_{ts}y_s}_{\cb_p}^2,$$
and the result then comes from the property (\ref{regu-hold-semi}).

\end{proof}

We are now in position to justify the use of (\ref{int-prov}) as a definition for the integral:

\begin{proposition}\label{defi-rig-int}
Let $y\in \hat{\cq}_{\ka,p}^\ka([0,T])$ admitting the decomposition $\delha y=X^{x,i}y^{x,i}+y^\sharp$, with $\ka \in (1/3,\ga)$ and $p\in \N^\ast$ such that $\ga-\ka > n/(2p)$. Assume that $f=(f_1,\ldots,f_N)$ with $f_i \in \mathcal{X}_2$ for $i=1,\ldots,N$. We set, for all $s<t$,
\begin{equation}\label{dec-int-rough}
\cj_{ts}(\hat{d}x \, f(y))=X^{x,i}_{ts}f_i(y_s)+X^{xa,i}_{ts}(y,f'_i(y))_s+X^{xx,ij}_{ts}(y^{x,j}\cdot  f'_i(y))_s+\Laha_{ts}(J),
\end{equation}
where we recall our Notation \ref{not:f-prime} for $f'_i$, and with
\begin{multline}
J_{tus}=X^{x,i}_{tu}(f_i(y)^\sharp_{us})+X^{xa,i}_{tu}((\delha y)_{us} ,f_i'(y_s))+X^{xa,i}_{tu}(y_u , \der(f'_i(y))_{us})\\
+X^{xx,ij}_{ts} \der(y^{x,j} \cdot f'_i(y))_{us},
\end{multline}
the term $f(y)^\sharp$ being defined by the decomposition (\ref{decomp-si-i-f}).
Then one has:
\begin{enumerate}
\item $\cj(\hat{d}x \, f(y))$ is well-defined and there exists $z\in \cq_{\ka,p}^\ka([0,T])$ such that $\delha z$ is equal to the increment $\cj(\hat{d}x \, f(y))$. Furthermore, for any $0\le s<t\le T$, the integral $\cj_{ts}(\hat{d}x \, f(y))$ coincides with a Riemann type integral for two regular functions $x$ and $y$.
\item The following estimation holds true
\begin{equation}\label{estim-int-rough}
\cn[z;\cq_{\ka,p}^\ka([0,T])] \leq \cfx \lcl 1+\cn[y;\cac_1^0(\cb_{\ka,p})]^2+T^\al \cn[y;\cq_{\ka,p}^\ka]^2 \rcl,
\end{equation}
for some $\al >0$.
\item For all $s<t$,
\begin{multline}\label{dec-int-rough-lim}
\cj_{ts}(\hat{d}x \, f(y))=\lim_{|\Delta_{[s,t]}| \to 0} \sum_{(t_k)\in \Delta_{[s,t]}} \big\{ X^{x,i}_{t_{k+1}t_k}f_i(y_{t_k}) +X^{xa,i}_{t_{k+1}t_k}(y,f'_i(y))_{t_k}\\
+X^{xx,ij}_{t_{k+1}t_k}(y^{x,j},f_i'(y))_{t_k} \big\},
\end{multline}
where the limit is taken over partitions $\Delta_{[s,t]}$ of the interval $[s,t]$, as their mesh tends to 0.
\end{enumerate}
\end{proposition}

\begin{proof}
The fact that $\cj_{ts}(\hat{d}x \, f(y))$ coincides with a Riemann type integral for two regular functions $x$ and $y$ is just what has been derived at equation (\ref{int-prov}). As far as the second claim of our proposition is concerned, it is a direct consequence of Hypotheses~\ref{def:reg-f} and~\ref{hypo-regu}, together with the estimations of Lemma \ref{lem-dec-f}. Let us check for instance the regularity of $J$:
\begin{itemize}
\item for $X^{x,i}(f_i(y)^\sharp)$, we get, by (\ref{regu-x-x-rough}) and (\ref{estim-f-sha-2-3}), 
$$\cn[X^{x,i}( f_i(y)^{\sharp,2}+f_i(y)^{\sharp,3} );\cac_3^{\ga+2\ka}(\cb_p)] \leq \cfx \, \cn[y;\Qha^\ka_{\ka,p}],$$
while, owing to (\ref{regu-x-x-rough}) and (\ref{estim-f-sha-1}),
$$\cn[X^{x,i} f_i(y)^{\sharp,1};\cac_3^{\ga+2\ka-n/(2p)}(\cb_p)] \leq \cfx \lcl \cn[y;\cac_1^0(\cb_{\ka,p})]^2+\cn[y;\Qha^\ka_{\ka,p}]\rcl.$$
\item for $X^{xa,i}((\delha y),f_i'(y))$, the hypothesis (\ref{regu-x-x-a}) gives
$$\cn[X^{xa,i}((\delha y),f_i'(y));\cac_3^{\ga+2\ka-n/(2p)}(\cb_p)] \leq \cfx \, \cn[y;\cacha_1^\ka(\cb_{\ka,p})] \leq \cfx \, \cn[y;\Qha^\ka_{\ka,p}].$$
\item for $X^{xa,i}(y,\der(f'_i(y)))$, one has, by (\ref{regu-x-x-a}) again,
\bean
\lefteqn{\cn[X^{xa,i}(y, \der(f'_i(y)));\cac_3^{\ga+2\ka-n/(2p)}(\cb_p)]}\\
&\leq & \cfx \, \cn[y;\cac_1^0(\cb_{\ka,p})] \cn[y;\cac_1^\ka(\cb_p)]\\
& \leq & \cfx \, \cn[y;\cac_1^0(\cb_{\ka,p})] \lcl \cn[y;\cac_1^0(\cb_{\ka,p})]+\cn[y;\Qha^\ka_{\ka,p}]\rcl.
\eean
\item for $X^{xx,ij}\der(y^{x,j}\cdot  f'_i(y))$, we deduce from (\ref{regu-x-x-x}) that
\bean
\lefteqn{\cn[X^{xx,ij}\der(y^{x,j}\cdot  f'_i(y));\cac_3^{2\ga+\ka}(\cb_p)]}\\
&\leq & \cfx \lcl \cn[y^{x,j};\cac_1^\ka(\cb_p)] +\cn[y^{x,j};\cac_1^0(\cb_{\ka,p})] \cn[y;\cac_1^\ka(\cb_p)] \rcl \\
&\leq & \cfx \lcl 1+\cn[y;\cac_1^0(\cb_{\ka,p})]^2+\cn[y;\Qha^\ka_{\ka,p}]^2 \rcl.
\eean
\end{itemize}
Moreover, thanks to the algebraic relations stated in Hypothesis \ref{hypo-regu} and the decomposition~(\ref{decomp-si-i-f}), it is easy to show that
$$J=-\delha \lp X^{x,i}(f_i(y))+X^{xa,i}(y,f_i'(y))+X^{xx,ij} (y^{x,j}\cdot  f'_i(y)) \rp.$$
Therefore, $J \in \text{Ker} \, \delha \cap \cac_3^\mu(\cb_p)$, with $\mu=\ga+2\ka-n/(2p) >1$, and we are allowed to apply $\Laha$. Besides, using the contraction property (\ref{contraction-laha}), we get
$$\cn[\Laha(J);\cac_2^{\ga+2\ka-n/(2p)}(\cb_p)] \leq \cfx \lcl 1+\cn[y;\cac_1^0(\cb_{\ka,p})]^2+\cn[y;\Qha^\ka_{\ka,p}]^2 \rcl,$$
and also
$$\cn[\Laha(J);\cac_2^{\ga+\ka-n/(2p)}(\cb_{\ka,p})] \leq \cfx \lcl 1+\cn[y;\cac_1^0(\cb_{\ka,p})]^2+\cn[y;\Qha^\ka_{\ka,p}]^2 \rcl.$$
The regularity of the other terms of (\ref{dec-int-rough}) can be proved with similar arguments. As for the expression (\ref{dec-int-rough-lim}), it is a consequence of Proposition \ref{lien-laha-int}, since one can write
$$\cj(\hat{d}x \, f(y))=\lp \id-\Laha \delha \rp \lp X^{x,i}(f_i(y))+X^{xa,i}(y,f_i'(y))+X^{xx,ij} (y^{x,j}\cdot f'_i(y)) \rp.$$

\end{proof}

\smallskip

Once our integral for controlled processes is defined, the existence and uniqueness of a local solution for our equation is easily proved:
\begin{theorem}\label{theo-cas-rough} Assume that $f=(f_1,\ldots,f_N)$ with $f_i \in \mathcal{X}_3$ for $i=1,\ldots,N$.
For any pair $(\ka,p) \in (1/3,\ga) \times \N$ such that $\ga-\ka >n/(2p)$, there exists a time $T >0$ for which the system
\begin{equation}\label{formu-ori}
(\delha y)_{ts}=\cj_{ts}(\hat{d}x \, f(y)) \quad , \quad y_0=\psi \in \cb_p,
\end{equation}
interpreted with Proposition \ref{defi-rig-int}, admits a unique solution $y$ in $\cq_{\ka,p}^\ka([0,T])$. 
\end{theorem}

\begin{proof}
This local solution is obtained via a standard fixed-point argument in the space of controlled processes. The procedure essentially leans on the estimation
(\ref{estim-int-rough}). The interested reader can refer to \cite{NNT} for further details on the principle of the proof.

\end{proof}

\section{Global solution under stronger regularity assumptions}
\label{sec:glob-sol}

The aim of this section is to show that a regularization in the nonlinearity involved in our heat equation can yield a global solution. Specifically,
this section is devoted to the proof of the existence and uniqueness of a global solution to the (slightly) modified system
\begin{equation}\label{eq:reg-system}
(\delha y)_{ts}=\int_s^t S_{tu} \, dx_u^{(i)} \, S_\ep f_i(y_u) \quad , \quad y_0=\psi,
\end{equation}
where $f_i \in \mathcal{X}_3$, $\psi \in \cb_{\al,p}$ for some $\al \geq 0$ to be precised, and $\ep$ is a strictly positive fixed parameter. Owing to the regularizing effect of $S_\ep$, we will see that such a system is much easier to handle than the original formulation (\ref{formu-ori}).

\smallskip

Note that we have chosen a regularization by $S_\ep$ in (\ref{eq:reg-system}), in order to be close to Teichmann's framework \cite{Tei}. However, it will be clear from the considerations below that an extension to a convolutional nonlinearity of the form
$$
[\tilde f_i(y)](\xi)=\int_{\R^n} K(\xi,\eta) \, f_i(y(\eta)) \, d\eta, \quad \xi\in\R^n,
$$
with a smooth enough kernel $K$, is possible.
The technical argument which enables to extend the local solution into a global one are taken from a previous work of two of the authors \cite{DTbis}.

\subsection{Heuristic considerations}

The regularizing property (\ref{regu-prop-semi}) of the semigroup $S_\ep$ allows us to turn to a decomposition of $\int_s^t S_{tu} \, dx_u^{(i)} \, S_\ep f_i(y_u)$ similar to the finite-dimensional case, or otherwise stated written without the help of the mixed operator $X^{xa}$. Indeed, let us go back to the decomposition (\ref{decompo-f-un}):
\begin{multline}\label{decompo-f-un-glo}
\der(f_i(y))_{ts}=(\der x)_{ts}y^x_s \cdot f_i'(y_s)
+\bigg[ a_{ts}y_s \cdot f'_i(y_s) \\+y^\sharp_{ts} \cdot f_i'(y_s)+(X^{ax,i}_{ts} y^{x,i}_s) \cdot f_i'(y_s) +\int_0^1 dr \, \lc f_i'(y_s+r(\der y)_{ts})-f_i'(y_s) \rc \cdot (\der y)_{ts}\bigg],
\end{multline}
but this time, let us consider the whole term into brackets as a remainder term evolving in $\cb_p$ (or maybe $\cb_{p/2}$), and denote it by $f_i(y)^\sharp$. This point of view is for instance justified if we let the process $y$ evolve in $\cb_{1,p}$, insofar as, for any $s,t \in I$,
$$\norm{a_{ts} y_s \cdot f_i'(y_s)}_{\cb_p} \lesssim \lln t-s \rrn \norm{f_i'}_\infty \norm{y_s }_{\cb_{1,p}} \lesssim \lln t-s \rrn^{2\ka} \lln I \rrn^{1-2\ka} \norm{f_i'}_\infty \norm{y_s }_{\cb_{1,p}}.$$
For obvious stability reasons, the strong assumption $y_s \in \cb_{1,p}$ then implies that the residual term steming from the decomposition of $\int_s^t S_{tu} \, dx_u^{(i)} \, S_\ep f_i(y_u)$ should also be seen as an element of $\cb_{1,p}$. This is made possible through the action of $S_\ep$. Indeed, owing to (\ref{regu-prop-semi}), one has
$$\norm{S_\ep(f(y)^\sharp)}_{\cb_{1,p}} \leq c \, \ep^{-1}  \,  \norm{f(y)^\sharp}_{\cb_p}, \quad \text{for some constant} \ c>0.$$

\subsection{Definition of the integral}
According to the above considerations, only the processes $X^{x,i}$, $X^{ax,i}$ and $X^{xx,i}$ will come into play. Therefore, let us focus on the following simplified version of Hypothesis \ref{hypo-regu}:

\begin{hypothesis}\label{hypo-regu-glo}
We assume that the process $x$ allows to define operators $X^{x,i}$, $X^{ax,i}$, $X^{xx,ij}$ ($i,j \in \{1,\ldots, N \}$), such that, recalling our Notation \ref{not:bilinear-B}:
\begin{itemize}
\item[(H1)] From an algebraic point of view:
\begin{equation}\label{hypo-alg-rough-x-x-glo}
\delha X^{x,i}=0
\end{equation}
\begin{equation}\label{hypo-alg-rough-x-ax-glo}
X^{x,i}=X^{ax,i}+\der x^i
\end{equation}
\begin{equation}\label{hypo-alg-rough-x-xx-glo}
\delha X^{xx,ij}=X^{x,i}(\der x^j).
\end{equation}
\begin{equation}\label{hypo-s-ep}
\text{The operators $X^{x,i}$ and $X^{xx,ij}$ commute with $S_\ep$.}
\end{equation}
\item[(H2)] From an analytical point of view: 
\begin{equation}\label{regu-x-x-rough-glo}
X^{x,i} \in \cac_2^\ga(\cl(\cb_p,\cb_p))  \cap \cac_2^\ga(\cl(\cb_{1,p},\cb_{1,p}))\cap \cac_2^{\ga-n/(2p)}(\cl(\cb_{p/2},\cb_p)) 
\end{equation}
\begin{equation}\label{regu-x-a-x-glo}
X^{ax,i}\in \cac_2^{1+\ga}(\cl(\cb_{1,p},\cb_p))
\end{equation}
\begin{equation}\label{regu-x-x-x-glo}
X^{xx,ij} \in \cac_2^{2\ga}(\cl(\cb_p ,\cb_p)) \cap \cac_2^{2\ga}(\cl(\cb_{1,p} ,\cb_{1,p})) .
\end{equation}
\end{itemize}

\end{hypothesis}

\begin{remark}
The assumption (\ref{hypo-s-ep}) is trivially met when $x$ is a differentiable process and $X^{x,i}$ is defined by $X^{x,i}_{ts}=\int_s^t S_{tu} \, dx_u^{(i)}$. It will remain true in rough cases, following the constructions of Section \ref{section-appli-mbf-13}. This commutativity property will be resorted to in the proofs of Propositions \ref{defi-int-glo} and \ref{contr-arg-glo}.
\end{remark}

The notion of controlled processes which has been introduced in Definition \ref{defi-contr-proc} can also be simplified in this context:

\begin{definition}
For any $\ka < \ga$, let us define the space
$$\tilde{Q}_{\ka,p}=\lcl y \in \cac_1^\ga(\cb_{1,p}): \ (\delha y)_{ts}=X^{x,i}_{ts} y^{x,i}_s+y^\sharp_{ts}, \ y^{x,i} \in \cac_1^\ka(\cb_{1,p}) \cap \cac_1^0(\cb_{1,p}), \ y^\sharp \in \cac_2^{2\ka}(\cb_{1,p}) \rcl,$$
together with the seminorm 
$$\cn[y;\tilde{Q}_{\ka,p}]=\cn[y^{x,i};\cac_1^0(\cb_{1,p})]+\cn[y^{x,i};\cac_1^\ka(\cb_{1,p})]+\cn[y^\sharp;\cac_2^{2\ka}(\cb_{1,p})].$$
\end{definition}

With this notation, one has $\cn[y;\cac_1^\ga(\cb_{1,p})] \leq c_x \, \cn[y;\tilde{Q}_{\ka,p}]$.

\smallskip

In the following two propositions, let us fix an interval $I=[a,b]$ and denote $\lln I \rrn=b-a$.

\begin{proposition}\label{defi-int-glo}
Let $y\in \tilde{Q}_{\ka,p}(I)$ with decomposition $\delha y=X^{x,i} y^{x,i}+y^\sharp$, for some $(\ka,p) \in (1/3,\ga ) \times \N^\ast$ such that $\ga -\ka > n/(2p)$ and initial value  $h=y_a \in \cb_{1,p}$. For any $\psi \in \cb_{1,p}$, define a process $z$ by the two relations: $z_a=\psi$ and for any $s<t \in I$,
\begin{multline*}
(\delha z)_{ts}=\cj_{ts}(\hat{d}x^{(i)} \, S_\ep f_i(y_s))=X^{x,i}_{ts} S_\ep f_i(y_s)+X^{xx,ij}_{ts} S_\ep(y^{x,j}_s \cdot f_i'(y_s))\\
+\Laha_{ts} \lp X^{x,i} S_\ep f_i(y)^\sharp+X^{xx,ij}S_\ep \der (y^{x,j} \cdot f_i'(y)) \rp,
\end{multline*}
where $f_i(y)^\sharp$ stands for the term into brackets in (\ref{decompo-f-un-glo}). Then:
\begin{itemize}
\item $z$ is well-defined as an element of $\tilde{Q}_{\ka,p}(I)$.
\item The following estimation holds:
\begin{equation}\label{contr-glo}
\cn[z;\tilde{Q}_{\ka,p}(I)] \leq c \, \ep^{-1} \lcl 1+\lln I \rrn^{2(\ga-\ka)} \cn[y;\tilde{Q}_{\ka,p}(I)]^2+\lln I \rrn^{2(1-\ka)} \norm{h}_{\cb_{1,p}}^2 \rcl,
\end{equation}
for some constant $c>0$.
\item For any $s<t \in I$, $(\delha z)_{ts}$ can also be written as
\begin{equation}\label{sums-int-glo}
(\delha z)_{ts}=\lim_{\lln \cp_{[s,t]} \rrn \to 0} \sum_{t_k \in \cp_{[s,t]}} \lcl X^{x,i}_{t_{k+1}t_k} S_\ep f_i(y_{t_k})+X^{xx,ij}_{t_{k+1}t_k} S_\ep \lp y^{x,j}_{t_k} \cdot f_i'(y_{t_k}) \rp \rcl \quad \text{in} \ \cb_{1,p}.
\end{equation}
\end{itemize}
\end{proposition}

\begin{proof}
Let us focus on the estimation of the residual term
$$z^\sharp_{ts}=X^{xx,ij}_{ts} S_\ep(y^{x,j}_s \cdot f_i'(y_s))
+\Laha_{ts} \lp X^{x,i} S_\ep f_i(y)^\sharp+X^{xx,ij}S_\ep \der (y^{x,j} \cdot f_i'(y)) \rp .$$
First, using (\ref{regu-x-x-x-glo}) and (\ref{regu-prop-semi}), we get
\bean
\norm{X^{xx,ij}_{ts} S_\ep (y^{x,j}_s \cdot f_i'(y_s))}_{\cb_{1,p}} &\leq & c_x \lln t-s \rrn^{2\ga} \ep^{-1} \norm{y^{x,j}_s \cdot f_i'(y_s)}_{\cb_p}\\
&\leq & c_x \lln t-s \rrn^{2\ga} \ep^{-1} \norm{y_s^{x,j}}_{\cb_{1,p}}\\
&\leq & c_x \lln t-s \rrn^{2\ga} \ep^{-1} \, \cn[y;\tilde{Q}_{\ka,p}(I)].
\eean

Secondly, write $f_i(y)^\sharp=f_i(y)^{\sharp,1}+f_i(y)^{\sharp,2}$, with $f_i(y)^{\sharp,1}_{ts}=a_{ts}y_s \cdot f'_i(y_s)+y^\sharp_{ts} \cdot f_i'(y_s)+(X^{ax,i}_{ts} y^{x,i}_s) \cdot f_i'(y_s)$, $f_i(y)^{\sharp,2}_{ts}=\int_0^1 dr \, [f_i'(y_s+r(\der y)_{ts})-f_i'(y_s)] \cdot (\der y)_{ts}$, and notice that
\bean
\lefteqn{\norm{X^{x,i}_{tu} S_\ep f_i(y)^{\sharp,1}_{us} }_{\cb_{1,p}}}\\
 &\lesssim & \lln t-u \rrn^\ga \ep^{-1} \norm{f_i(y)^{\sharp,1}_{us}}_{\cb_p}\\
& \lesssim & \lln t-u \rrn^\ga \ep^{-1} \lcl \norm{(a_{us}y_s) \cdot f_i'(y_s)}_{\cb_p}+\norm{(X^{ax,i}_{us} y^{x,i}_s) \cdot f_i'(y_s)}_{\cb_p}+\norm{y^\sharp_{us} \cdot f_i'(y_s)}_{\cb_p} \rcl\\
&\lesssim & \lln t-u \rrn^\ga \ep^{-1} \lcl \lln u-s \rrn \norm{y_s }_{\cb_{1,p}}+\lln u-s \rrn^{1+\ga} \norm{y^{x,i}_s }_{\cb_{1,p}}+\norm{y^\sharp_{us}}_{\cb_{1,p}} \rcl \\
&\lesssim & \lln t-u \rrn^\ga \ep^{-1} \lcl \lln u-s \rrn^{2\ka} \cn[y;\tilde{Q}_{\ka,p}(I)]+\lln u-s \rrn \lcl \cn[y;\tilde{Q}_{\ka,p}(I)]+\norm{h}_{\cb_{1,p}} \rcl \rcl \\
& \lesssim & \lln t-s \rrn^{\ga+2\ka} \ep^{-1} \lcl \cn[y;\tilde{Q}_{\ka,p}(I)]+\lln I \rrn^{1-2\ka} \norm{h}_{\cb_{1,p}} \rcl,
\eean

while, owing to (\ref{hypo-s-ep}),
\bean
\lefteqn{\norm{X^{x,i}_{tu} S_\ep f_i(y)^{\sharp,2}_{us}}_{\cb_{1,p}} \ = \ \norm{S_\ep X^{x,i}_{tu}  f_i(y)^{\sharp,2}_{us}}_{\cb_{1,p}} }\\
&\lesssim & \ep^{-1} \lln t-u \rrn^{\ga-n/(2p)} \norm{f_i(y)^{\sharp,2}_{us}}_{\cb_{p/2}}\\
&\lesssim & \ep^{-1} \lln t-u \rrn^{\ga-n/(2p)} \norm{(\der y)_{us}}_{\cb_{p}}^2\\
&\lesssim & \ep^{-1} \lln t-u \rrn^{\ga-n/(2p)} \lcl \norm{(\delha y)_{us}}_{\cb_{p}}^2 +\norm{a_{us}y_s }_{\cb_p}^2  \rcl \\
&\lesssim & \ep^{-1} \lln t-u \rrn^{\ga-n/(2p)} \lcl \lln u-s \rrn^{2\ga} \cn[y;\tilde{Q}_{\ka,p}(I)]^2 +\lln u-s \rrn^2 \lcl \cn[y;\tilde{Q}_{\ka,p}(I)]^2+\norm{h}_{\cb_{1,p}}^2 \rcl \rcl \\
&\lesssim & \ep^{-1} \lln t-s \rrn^{3\ga-n/(2p)} \lcl \cn[y;\tilde{Q}_{\ka,p}(I)]^2+\lln I \rrn^{2(1-\ga)} \norm{h}_{\cb_{1,p}}^2 \rcl.
\eean

Even more simple estimations based on (\ref{regu-x-x-x-glo}) give
\begin{multline*}
\norm{X^{xx,ij}_{tu} S_\ep \der(y^{x,j} \cdot f_i'(y))_{us} }_{\cb_{1,p}}\\
 \lesssim \ep^{-1} \lln t-s \rrn^{2\ga +\ka} \lcl 1+\cn[y;\tilde{Q}_{\ka,p}(I)]^2+\lln I \rrn^{1-\ka} \cn[y;\tilde{Q}_{\ka,p}(I)]\cdot  \norm{h}_{\cb_{1,p}} \rcl.
\end{multline*}

Thanks to the contraction property (\ref{contraction-laha}), we now easily deduce
$$\cn[z^\sharp;\cac_2^{2\ka}(I)] \leq c \, \ep^{-1} \lcl 1+\lln I \rrn^{2(\ga-\ka)} \cn[y;\tilde{Q}_{\ka,p}(I)]^2+\lln I \rrn^{2(1-\ka)} \norm{h}_{\cb_{1,p}}^2 \rcl.$$

The estimation of $\cn[z^{x,i};\cac_1^{0,\ka}(I;\cb_{1,p})]$ can be established along the same lines. As for (\ref{sums-int-glo}), it is a consequence of (\ref{lien-laha-int}), together with the reformulation 
$$\delha z =(\id -\Laha \delha)(X^{x,i}S_\ep f_i(y)+X^{xx,ij} S_\ep(y^{x,j} \cdot f_i'(y))).$$

\end{proof}

In order to settle an efficient fixed-point argument in this context, the following Lipschitz relation is required:

\begin{proposition}\label{contr-arg-glo}
If $y,\yti \in \tilde{Q}_{\ka,p}(I)$ with $y_a=\yti_a$, and if we denote by $z,\zti$ the two processes in $\tilde{Q}_{\ka,p}(I)$ such that
$$z_0=\zti_0=y_0 \quad \text{and} \quad \delha z=\cj(\hat{d} x^{(i)} \, S_\ep f_i(y)) \ , \ \delha \zti=\cj(\hat{d} x^{(i)} \, S_\ep f_i(\yti)),$$
then
\begin{multline}\label{ineg-contr-glo}
\cn[z-\zti;\tilde{Q}_{\ka,p}(I)] \leq c_x \, \ep^{-1} \lln I \rrn^{\ga-\ka} \cn[y-\yti;\tilde{Q}_{\ka,p}(I)] \\ 
\lcl 1+\lln I \rrn^{2(\ga-\ka)} \{ \cn[y;\tilde{Q}_{\ka,p}(I)]^2 +\cn[y;\tilde{Q}_{\ka,p}(I)]^2 \} +\lln I \rrn^{2(1-\ka)} \norm{h}^2_{\cb_{1,p}} \rcl .
\end{multline}
\end{proposition}

\begin{proof}
One has, for any $s,t \in I$,
\begin{multline*}
\delha (z-\zti)_{ts} =X^{x,i}_{ts} S_\ep (f_i(y_s)-f_i(\yti_s))+X^{xx,ij}_{ts} S_\ep (y^{x,j}_s \cdot f_i'(y_s)-\yti^{x,j}_s \cdot f_i'(\yti_s) )\\
+\Laha_{ts} \lp X^{x,i} S_\ep (f_i(y)^\sharp -f_i(\yti)^\sharp)+X^{xx,ij} \der (y^{x,j} \cdot f_i'(y)-\yti^{x,j} \cdot f_i'(\yti)) \rp.
\end{multline*}

Let us only focus on the more intricate term, that is to say $X^{x,i}S_\ep (f_i(y)^{\sharp,2}-f_i(\yti)^{\sharp,2})$, where, according to the notations of the proof of Proposition \ref{defi-int-glo}, 
$$f_i(y)^{\sharp,2}_{ts}=\int_0^1 dr \, \lc f_i'(y_s+r(\der y)_{ts})-f_i'(y_s) \rc \cdot (\der y)_{ts}.$$
Write
\begin{multline*}
f_i(y)^{\sharp,2}_{ts}-f_i(\yti)^{\sharp,2}_{ts}=\int_0^1 dr \, \lc f_i'(y_s+r(\der y)_{ts})-f_i'(y_s)\rc \cdot \der (y-\yti)_{ts}\\
+(\der \yti)_{ts} \cdot \der (y-\yti)_{ts} \cdot \int_0^1 dr \, r \int_0^1 dr' \, f_i''(y_s+rr' (\der y)_{ts})\\
+(\der \yti)^2_{ts} \cdot \int_0^1 dr \, r \int_0^1 dr' \, \lc f_i''(y_s+rr'(\der y)_{ts})-f_i''(\yti_s+rr' (\der \yti)_{ts}) \rc.
\end{multline*}

In this way,
\begin{multline*}
\norm{f_i(y)^{\sharp,2}_{ts}-f_i(\yti)^{\sharp,2}_{ts}}_{\cb_{p/2}} \lesssim \norm{\der(y-\yti)_{ts}}_{\cb_p} \lcl \norm{(\der y)_{ts}}_{\cb_p}+\norm{(\der \yti)_{ts}}_{\cb_p} \rcl \\
+\norm{(\der \yti)_{ts}}_{\cb_p}^2 \lcl \norm{y_s-\yti_s}_{\cb_\infty}+\norm{y_t-\yti_t}_{\cb_\infty} \rcl.
\end{multline*}

Now
\bean
\norm{\der (y-\yti)_{ts}}_{\cb_p} &\lesssim & \norm{\delha (y-\yti)_{ts}}_{\cb_{1,p}}+\lln t-s \rrn \norm{(y_s-\yti_s)-S_{sa}(y_a-\yti_a)}_{\cb_{1,p}} \\
&\lesssim & \lln t-s \rrn^\ga \cn[y-\yti;\tilde{Q}_{\ka,p}(I)], 
\eean

while

\bean
\norm{(\der y)_{ts}}_{\cb_p} &\leq & \norm{(\delha y)_{ts}}_{\cb_{1,p}}+\norm{a_{ts}(\delha y)_{sa}}_{\cb_p}+\norm{a_{ts}S_{sa}h}_{\cb_p}\\
&\lesssim & \lln t-s \rrn^\ka \lcl \lln I \rrn^{\ga-\ka} \cn[y;\tilde{Q}_{\ka,p}(I)]+\lln I \rrn^{1-\ka} \norm{h}_{\cb_{1,p}} \rcl
\eean

and finally

\bean
\norm{y_s-\yti_s }_{\cb_\infty} \ \lesssim \ \norm{y_s-\yti_s}_{\cb_{1,p}} &\lesssim & \norm{y_s-\yti_s-S_{sa}(y_a-\yti_a)}_{\cb_{1,p}}\\
&\lesssim & \lln I \rrn^{\ga-\ka} \cn[y-\yti;\tilde{Q}_{\ka,p}(I)].
\eean

This easily leads to
\begin{multline*}
\cn[f_i(y)^{\sharp,2}-f_i(\yti)^{\sharp,2};\cac_2^{2\ka}(\cb_{p/2})]
\lesssim \lln I \rrn^{\ga-\ka} \cn[y-\yti;\tilde{Q}_{\ka,p}(I)] \\ \lcl 1+\lln I\rrn^{2(\ga-\ka)} \lcl \cn[y;\tilde{Q}_{\ka,p}(I)]^2+\cn[\yti;\tilde{Q}_{\ka,p}(I)]^2 \rcl +\lln I \rrn^{2(1-\ka)} \norm{h}_{\cb_{1,p}}^2 \rcl.
\end{multline*}

Inequality (\ref{ineg-contr-glo}) now follows from standard computations based on Hypothesis \ref{hypo-regu-glo}.

\end{proof}

We are now in position to state the expected global result:

\begin{theorem}
Let $f_i \in \mathcal{X}_3$, for $i\in \{1,\ldots, N \}$. Under Hypothesis \ref{hypo-regu-glo}, let $(\ka,p) \in (1/3,\ga) \times \N^\ast$ such that $\ga-\ka >n/(2p)$. For any $T>0$, for any $\psi \in \cb_{1,p}$, the differential system
$$(\delha y)_{ts}=\cj_{ts}(\hat{d} x^{(i)} \, S_\ep f_i(y)) \quad , \quad y_0=\psi,$$
interpreted with Proposition \ref{defi-int-glo}, admits a unique global solution in $\tilde{Q}_{\ka,p}([0,T])$.

\end{theorem}

\begin{proof}
From the two estimations (\ref{contr-glo}) and (\ref{ineg-contr-glo}), the patching argument is exactly the same as in \cite[Theorem 4.16]{DTbis}. It consists in controlling both the norm of the initial value and the norm of the process as a controlled path on each successive intervals. For sake of conciseness, the reader is refered to the latter article for a detailed proof of the statement.
\end{proof}

\section{Application}\label{section-appli-mbf-13}

As it was announced in the introduction, the goal here is to apply the previous abstract results of both Sections \ref{section-cas-young} and \ref{section-cas-rough} to a fractional non linearity given by the formula
\begin{equation}\label{forme-bruit-appli}
X_t(\varphi)(\xi)=\sum_{i=1}^N  \, x^{i}_t \sigma_i(\xi,\varphi(\xi)),
\end{equation}
with a $d$-dimensional $\ga$-Hölder process $x=(x^{1},\ldots,x^{N})$  with $\ga >1/3$, and $\sigma_i$ some smooth elements of $\cx_2$, as defined in Hypothesis \ref{def:reg-f}.

\smallskip

To this end, we know that it suffices to construct, from $x$, a path $\textbf{X}=(X^{x},X^{ax},X^{xa},$ $X^{xx})$ which satisfies Hypothesis \ref{hypo-regu}. Indeed, the latter assumption clearly covers Hypothesis~\ref{hypo-x-x-young} of Section \ref{section-cas-young}.\smallskip

As usual in this paper, we shall proceed in two steps: we first work at a heuristic level, that is with smooth processes, and try to obtain an expression which can be extended to irregular situations. We then check directly Hypothesis \ref{hypo-regu} on the expression obtained in the heuristic step.

\subsection{Heuristic considerations}
Assume for the moment that $x$ is a smooth $\R^N$-valued function. Then the operators $X^{x},X^{ax},X^{xa}$ and $X^{xx}$ are defined by the formulae
\begin{equation}\label{def-abs-1}
X^{x,i}_{ts}(\vp)(\xi)=\int_s^t S_{tu}(\vp)(\xi) \, dx^{i}_u \quad , \quad X^{ax,i}_{ts}(\vp)(\xi)=\int_s^t a_{tu}(\vp)(\xi) \, dx^{i}_u,
\end{equation}
\begin{equation}\label{def-abs-2}
X^{xa,i}_{ts}(\vp,\psi)(\xi)=\int_s^t S_{tu}((a_{us}\vp) \cdot \psi)(\xi) \, dx^{i}_u
\end{equation}
\begin{equation}\label{def-abs-3}
X^{xx,ij}_{ts}(\vp)(\xi) =\int_s^t S_{tu}(\vp )(\xi) \, dx^{i}_u \, (\der x^{j})_{us}.
\end{equation}
Set now $x^2_{ts}=\int_s^t dx_u \otimes (\der x)_{us}$. Then a straightforward integration by parts argument yields the following expression for the increments introduced above:
\begin{eqnarray}
X^{x,i}_{ts}&=&(\der x^{i})_{ts}+\int_s^t A S_{tu} (\der x^{i})_{us} \, du  \label{def-ipp-x-x} \\
X^{ax,i}_{ts}&=&\int_s^t A S_{tu} (\der x^{i})_{us} \, du \label{def-ipp-x-ax} \\
X^{xa,i}_{ts}&=&\int_s^t X^{x,i}_{tu} (AS_{us} \otimes \id) \, du  \label{def-ipp-x-xa} \\
X^{xx,ij}_{ts}&=&x^{2,ij}_{ts}+\int_s^t AS_{tu} x^{2,ij}_{us} \, du. \label{def-ipp-x-xx}
\end{eqnarray}

These are the expressions that we are ready to extend to irregular processes.

Let us only elaborate on how to get (\ref{def-ipp-x-xa}). Actually, it suffices to notice that
$$\int_s^t S_{tu}((a_{us}\vp) \cdot \psi) \, dx^i_u =-\int_s^t \partial_u(X^{x,i}_{tu})((a_{us}\vp) \cdot \psi),$$
where, in the last integral, the partial derivative $\partial_u$ only applies to the operato $X^{x,i}_{tu}$. Then
\bean
-\int_s^t \partial_u(X^{x,i}_{tu})((a_{us}\vp) \cdot \psi) &=& \lc -X^{x,i}_{tu}((a_{us}\vp) \cdot \psi)\rc_s^t+\int_s^t du \, X^{x,i}_{tu} (\partial_u(a_{us}\vp) \cdot \psi)\\
&=& \int_s^t du \, X^{x,i}_{tu}((\Delta S_{us} \vp) \cdot \psi).
\eean

\begin{remark}
At this point, it is not clear that the integral expressions $\int_s^t A S_{tu} (\der x^{i})_{us} \, du$,... give rise to operators defined on $\cb_{\al,p}$. For the moment, we only consider those expressions as operators acting on $\cac^\infty_c$. The extension to any space $\cb_{\al,p}$ will stem from a continuity argument (see the proof of Proposition \ref{prop-rough-op-rough-path}).
\end{remark}

\subsection{Definition of the heat equation rough path}

In a natural way, in order to extend expressions (\ref{def-ipp-x-x})-(\ref{def-ipp-x-xx}) to a Hölder path $x$, one has to suppose that this process generates a standard rough path, that is to say:

\begin{hypothesis}\label{hypo-x-rough-path}
We assume that $x$ allows to construct a process $x^2 \in \cac_2^{2\ga}(\R^n \otimes \R^n)$ such that $\der x^2= \der x \otimes \der x$, or in other words
\begin{equation}\label{aire-levy-standard}
(\der x^{2,ij})_{tus}= (\der x^{i})_{tu} (\der x^{j})_{us} \quad , \quad i,j=1,\ldots, N.
\end{equation}
\end{hypothesis}

This allows us to state the main result of the section:
\begin{proposition}\label{prop-rough-op-rough-path}
Under Hypothesis \ref{hypo-x-rough-path}, the operators $X^{x,i}$, $X^{ax,i}$, $X^{xa,i}$, $X^{xx,ij}$ defined by (\ref{def-ipp-x-x})-(\ref{def-ipp-x-xx}), can be extended to a path $\textbf{X}$ which satisfies Hypothesis \ref{hypo-regu}.
\end{proposition}

\begin{proof}
We have to check both the algebraic and analytic assumptions.

\smallskip

\noindent
\textit{Algebraic conditions}. The verification of (\ref{hypo-alg-rough-x-x})-(\ref{hypo-alg-rough-x-xx}) is a matter of elementary calculations. For instance, let us have a look at relation (\ref{hypo-alg-rough-x-xx}). For all $s<u<t$, one has
$$(\delha X^{xx,ij})_{tus}=x^{2,ij}_{ts}-x^{2,ij}_{tu}-S_{tu} x^{2,ij}_{us} +\int_u^t AS_{tv}(x^{2,ij}_{vs}-x^{2,ij}_{vu}) \,j dv.$$
Then, by (\ref{aire-levy-standard}), this expression reduces to
\bean
\lefteqn{(\delha X^{xx,ij})_{tus}}\\
&=& (\id -S_{tu}) x^{2,ij}_{us}+(\der x^{i})_{tu}(\der x^{j})_{us}+\int_u^t AS_{tv}(x^{2,ij}_{us}+(\der x^{i})_{vu}(\der x^{j})_{us}) \, dv\\
&=& \lc (\der x^{i})_{tu}+\int_u^t AS_{tv}(\der x^{i})_{vu} \, dv \rc (\der x^{j})_{us} \ = \ X^{x,i}_{tu} (\der x^{j})_{us}.
\eean

\smallskip

\noindent
\textit{Analytical conditions}. Let us examine the regularity of each operator individually.

\smallskip

\noindent
\textit{Case of $X^{x,i}$}. The norms at stake here are

\begin{equation}\label{regu-x-x-l-p}
\cn[X^{x,i} ; \cac_2^\ga(\cl(\cb_p,\cb_p)) ]
\end{equation}
\begin{equation}\label{regu-x-x-w-p}
\cn[X^{x,i} ;\cac_2^\ga(\cl(\cb_{\ka,p},\cb_{\ka,p})) ]
\end{equation}
\begin{equation}\label{regu-l-p/2}
\cn[X^{x,i} ;\cac_2^{\ga-n/2p}(\cl(\cb_{p/2},\cb_p)) ].
\end{equation}

In order to establish those regularity results, let us first rewrite (\ref{def-ipp-x-x}) as 
$$X^{x,i}_{ts}=S_{ts}(\der x^{i})_{ts}-\int_s^t AS_{tu} (\der x^{i})_{tu} \, du.$$

Then observe that (\ref{regu-x-x-l-p}) and (\ref{regu-x-x-w-p}) are obtained thanks to the same kind of arguments. We thus focus on (\ref{regu-x-x-w-p}) for sake of conciseness. But the latter norm can be bounded easily by noticing that:
\begin{align*}
&\|X^{x,i}_{ts}(\varphi) \|_{\cb_{\ka,p}} \le \|S_{ts}(\varphi) \|_{\cb_{\ka,p}} |(\der x^{i})_{ts}|+\int_s^t
\|A S_{tu}(\varphi) \|_{\cb_{\ka,p}} |(\der x^{i})_{tu}| du\\
&\lesssim \|\varphi \|_{\cb_{\ka,p}} \|x^{i}\|_{\gamma} \lp |t-s|^\gamma +\int_s^t
|t-u|^{-1+\gamma}du \rp \lesssim \|\varphi \|_{\cb_{\ka,p}} \|x^{i}\|_{\gamma} |t-s|^\gamma,
\end{align*}
which holds for all $\ka\ge 0$. Along the same lines, in order to prove (\ref{regu-l-p/2}), we use the fact that $\|S_{ts}(\varphi) \|_{\cb_{p}} \lesssim \|\varphi\|_{\cb_{p/2}} |t-s|^{-n/2p}$ and that $\|A S_{ts}(\varphi) \|_{\cb_{p}} \lesssim \|\varphi\|_{\cb_{p/2}} |t-s|^{-1-n/2p}$. Then we obtain 
$$
\|X^{x,i}_{ts}(\varphi) \|_{\cb_{p}} 
\lesssim \|\varphi \|_{\cb_{p/2}} \|x^{i}\|_{\gamma} |t-s|^{\gamma-n/2p}
$$
for all $p$ such that $\gamma-n/2p >0$. Those estimations give the required bound (\ref{regu-l-p/2}).

\smallskip

\noindent
\textit{Case of $X^{ax,i}$}. We should now check that (\ref{regu-x-a-x}) is verified in our setting. To this aim, write $X^{ax,i}$ as
$$X^{ax,i}_{ts}=a_{ts}(\der x^{i})_{ts}-\int_s^t AS_{tu} (\der x^{i})_{tu} \, du.$$
Then
$$
\|X^{ax,i}_{ts}(\vp)\|_{\cb_p}=\|a_{ts}(\vp)\|_{\cb_p} |(\der x^{i})_{ts}|+ \int_s^t \|A S_{tu}(\vp)\|_{\cb_p} |(\der x^{i})_{tu}| du
$$
and using the semigroup estimates
$$
\|a_{ts}(\vp)\|_{\cb_p} \lesssim \|\varphi\|_{\cb_{\ka,p}} |t-s|^{\ka}\qquad
\|A S_{tu}(\vp)\|_{\cb_p} \lesssim  \|\varphi\|_{\cb_{\ka,p}} |t-u|^{-1+\ka}
$$
we easily conclude that
\begin{equation}
\cn[X^{ax,i};\cl(\cb_{\ka,p},\cb_p)] \lesssim c_x \lln t-s \rrn^{\ga+\ka},
\end{equation}

which is the expected regularity result.

\smallskip

\noindent
\textit{Case of $X^{xa,i}$}. Going back to (\ref{regu-x-x-a}), one must prove that the following norms are finite:
\begin{equation}\label{x-x-a-ka-p-p-p}
\cn[X^{xa,i};\cac_2^{\ga+\ka-n/(2p)}(\cl(\cb_{\ka,p} \times \cb_p,\cb_p))],
\quad\mbox{and}\quad
\cn[X^{xa,i};\cac_2^\ga(\cl(\cb_{\ka,p} \times \cb_{\ka,p},\cb_{\ka,p}))].
\end{equation}
To do so, write $X^{xa,i}_{ts}$ as
$$X^{xa,i}_{ts}=X^{x,i}_{ts}(a_{ts} \otimes \id)-\int_s^t S_{tu} X^{x,i}_{us} \lp AS_{us} \otimes \id \rp \, du.$$
We deduce
$$
\cn[X^{xa,i}_{ts}(\vp,\psi);\cb_p] \lesssim \cn[X^{x,i};\cac_2^{\ga-n/(2p)}(\cl(\cb_{p/2},\cb_{p}))] \cn[((a_{ts}\vp) \cdot \psi); \cb_{p/2}] 
$$
$$
\qquad + \cn[X^{x,i};\cac_2^{\ga-n/(2p)}(\cl(\cb_{p/2},\cb_p))] \int_s^t  |u-s|^{\ga}\cn[((A S_{us}\vp) \cdot \psi); \cb_{p/2}] du 
$$
where 
$$
 \cn[((a_{ts}\vp) \cdot \psi); \cb_{p/2}]
  \lesssim \cn[a_{ts}\vp; \cb_{p}] \cn[\psi; \cb_{p}]
\lesssim  |t-s|^{\ka}\cn[\vp; \cb_{\ka,p}] \cn[\psi; \cb_{p}]
$$ and
$$
\cn[((A S_{us}\vp) \cdot \psi); \cb_{p/2}]  \lesssim \cn[A  S_{us}\vp ; \cb_p] \cn[ \psi; \cb_p] 
\lesssim |u-s|^{-1+\ka} \cn[\vp ; \cb_{\ka,p}] \cn[ \psi; \cb_p] . 
$$
This allows to conclude that
\begin{multline*}
\cn[X^{xa,i}_{ts}(\vp,\psi);\cb_p]\\ 
\lesssim \cn[X^{x,i};\cac_2^{\ga-n/(2p)}(\cl(\cb_{p/2},\cb_p))]  \cn[\vp ; \cb_{\ka,p}] \cn[ \psi; \cb_p] |t-s|^{\ga+\ka-n/(2p)},
\end{multline*}
and the first of the required bounds in (\ref{x-x-a-ka-p-p-p}) follows. For the second one, we have
\begin{multline*}
\cn[X^{xa,i}_{ts}(\vp,\psi);\cb_{\ka,p}] \lesssim \cn[X^{x,i};\cac_2^{\ga}(\cl(\cb_{\ka,p},\cb_{\ka,p}))] \cn[((a_{ts}\vp) \cdot \psi); \cb_{\ka,p}] 
\\
 + \cn[X^{x,i};\cac_2^{\ga}(\cl(\cb_{\ka,p},\cb_{\ka,p}))] \int_s^t  |u-s|^{\ga}\cn[((A S_{us}\vp) \cdot \psi); \cb_{\ka,p}] du ,
\end{multline*}
and using the algebra property of $\cb_{\ka,p}$, we get
$$
\cn[((a_{ts}\vp) \cdot \psi); \cb_{\ka,p}]  \lesssim 
\cn[\vp; \cb_{\ka,p}] 
\cn[\psi; \cb_{\ka,p}] 
$$
and
$$
\cn[((A S_{us}\vp) \cdot \psi); \cb_{\ka,p}] 
\lesssim 
\cn[A  S_{us}\vp; \cb_{\ka,p}] \cn[ \psi; \cb_{\ka,p}]
\lesssim |u-s|^{-1}
\cn[\vp; \cb_{\ka,p}] \cn[ \psi; \cb_{\ka,p}]
$$
so that
\begin{multline*}
\cn[X^{xa,i}_{ts}(\vp,\psi);\cb_p] \\
\lesssim \cn[X^{x,i};\cac_2^{\ga}(\cl(\cb_{\ka,p},\cb_{\ka,p}))] 
\cn[\vp; \cb_{\ka,p}] \cn[ \psi; \cb_{\ka,p}]
(|t-s|^\gamma+\int_s^t  |u-s|^{\ga-1}du).
\end{multline*}
The second estimate follows.

\smallskip

\noindent
\textit{Case of $X^{xx,ij}$}. We must estimate the norm
\begin{equation}\label{eq:norm-X-xx}
\cn[X^{xx,ij};\cac_2^{2\ga}(\cl(\cb_p ,\cb_p))],
\end{equation}
and also $\cn[X^{xx,ij};\cac_2^{2\ga}(\cl(\cb_{\al,p} ,\cb_{\al,p}))]$ and $\cn[X^{xx,ij};\cac_2^{2\ga}(\cl(\cb_{\al,p} ,\cb_{p}))]$. We focus on (\ref{eq:norm-X-xx}), the others terms having similar behavior using the algebra property of $\cb_{\al,p}$ and the Sobolev embedding $\cb_{\al,p}\subset \cb_\infty$. 

First, write $X^{xx,ij}_{ts}$ as
$$X^{xx,ij}_{ts}=S_{ts} x^{2,ij}_{ts}-\int_s^t A S_{tu} \lc x^{2,ij}_{tu}+(\der x^{i})_{tu}(\der x^{j})_{us} \rc \, du.$$
From this expression, we immediately get
\bean
\lefteqn{\cn[X^{xx,ij}_{ts}(\vp);\cb_p]}\\
& \lesssim &  c_x \lcl \cn[ S_{ts}(\vp );\cb_p] | |t-s|^{2\gamma} +\int_s^t \cn[A S_{tu}(\vp );\cb_p]   [|t-u|^{2\gamma }+|t-u|^\gamma |u-s|^\gamma|] du \rcl \\
& \lesssim &  c_x \lcl \cn[ \vp ;\cb_p] | |t-s|^{2\gamma} + \cn[\vp ;\cb_p] \int_s^t \lln t-u \rrn^{-1}   [|t-u|^{2\gamma }+|t-u|^\gamma |u-s|^\gamma|] du \rcl \\
&\lesssim & c_x \, \cn[ \vp ;\cb_p] \lln t-s \rrn^{2\ga}.
\eean
This gives the expected conclusion $\cn[X^{xx,ij};\cac_2^{2\ga}(\cl(\cb_p ,\cb_p))]< \infty$.

\end{proof}

We are thus in position to apply the abstract Theorems \ref{theo-cas-young} and \ref{theo-cas-rough} in order to solve the heat equation for a general rough path above $x$:

\begin{theorem}
Let $x=(x^{(1)},\ldots,x^{(d)})$ a $d$-dimensional $\ga$-Hölder path ($\ga>1/3$) satisfying the rough path hypothesis \ref{hypo-x-rough-path}, and consider the infinite-dimensional noise $X$ build on $x$ through the formula (\ref{forme-bruit-appli}). If $f\in \cac^{3,b}(\R;\R)$, then the stochastic differential system
\begin{equation}\label{eq:heat-thm}
(\delha y)_{ts} =\int_s^t S_{tu} \, dX_u(y_u), \quad y_0 =\psi \in \cb_{\ka,p},
\end{equation}
interpreted with Proposition \ref{prop-int-young} if $\ga>1/2$ and Proposition \ref{defi-rig-int} if $\ga\in (1/3,1/2]$, admits:
\begin{itemize}
\item A unique global solution in $\cacha_1^{0,\ka}([0,T],\cb_{\ka,p})$ if $H>1/2$, where the pair $(\ka,p) \in (0,\ga) \times \N^\ast$ is such that $H+\ka >1$ and $2\ka p>1$.
\item A unique local solution in $\cq_{\ka,p}^\ka([0,T^\ast])$ if $\ga\in (1/3,1/2]$, where $T^\ast$ is a strictly positive random time and the pair $(\ka,p)\in (1/3,\ga) \times \N^\ast$ is such that $H-\ka >1/(2p)$.
\end{itemize}
\end{theorem}

\begin{remark}
It is a well-known fact that one can construct a rough path (in the sense of Hypothesis \ref{hypo-x-rough-path}) above a $N$-dimensional fractional Brownian motion $B$ with Hurst parameter $H>1/3$ (see e.g. \cite{CQ,FV,NNT,Un}). This means that we can solve the heat equation~(\ref{eq:heat-thm}) driven by this kind of process.
\end{remark}

\section{Rough case of order 3}\label{section-rough-case-2}

To conclude with, and also to reinforce the feeling that our approach to the problem (\ref{systeme-final}) is viable, let us say a few words about the case of a $\ga$-Hölder noise $x$, with $\ga \in (1/4,1/3]$ only. We will not present the construction of the integral with as many details as in the previous section, and will stick to the broad lines of the calculations.

\subsection{Construction of the integral} Fix an index $\ga \in (1/4,1/3]$ which represents the regularity of $x$. In order to be allowed to invert $\delha$ via Theorem \ref{existence-laha}, one must look here for a term of order 4, or more exactly of order $\ga+3\ka$, where $\ka$ is such that $\ka < \ga$ and $\ga+3\ka >1$. The crucial point of the following construction lies in the (obvious) existence of a coefficient $\ka \in (0,1/4)$ such that $\ga+3\ka >1$. Since $\ka <1/4$, we can resort to the space $\cb_{2\ka,p}$ and envisage the possibility of a solution evolving in this space. In this context, the operator $X^{xa,i}$ that we have introduced in the previous section, and which was formally defined as $X^{xa,i}_{ts}=\int_s^t S_{tu} \, dx^i_u \, (a_{us} \otimes \id)$, becomes an order-three operator:
$$X^{xa,i} \in \cac_2^{\ga+2\ka}(\cb_{2\ka,p} \times \cb_p,\cb_p).$$
Taking this observation into account, it seems then quite appropriate to consider the following space of controlled processes:
\begin{multline*}
\cq_{2\ka,p}^\ka=\{ y \in \cacha_1^\ka(\cb_{2\ka,p}): \, (\delha y)_{ts}=X^{x,i}_{ts}y^{x,i}_s+X^{xx,ij}_{ts}y^{xx,ij}_s+y^\sharp_{ts}, \\
(\der y^{x,i})_{ts}=(\der x^j)_{ts}\cdot  y^{xx,ji}_s +y^{x,\sharp,i}_{ts},\\
y^{x,i} \in \cac_1^0(\cb_{2\ka,p}) \cap \cac_1^\ka(\cb_p) \ , \ y^{xx,ij} \in \cac_1^0(\cb_{2\ka,p}) \cap \cac_1^\ka(\cb_p) \ , \ y^\sharp \in \cac_2^\ga(\cb_{2\ka,p}) \cap \cac_2^{3\ka}(\cb_p) \ ,\\
y^{x,\sharp,i} \in \cac_2^{2\ka}(\cb_p)  \},
\end{multline*}
together with its natural norm. One should notice the additional relation between $y^{x}$ and $y^{xx}$ which appears in the definition above, with respect to the rough case of order 2. This is reminiscent of the nilpotent algebra structure of \cite{FV}, and also of the algebraic structures introduced in \cite{g-ramif,g-abs,TT}.

\smallskip

According to our usual way to construct rough integrals, for the time being, $x$ is assumed to be differentiable and the operators $X^{x,i}$ and $X^{xx,ij}$ are defined by the formulae
$$X^{x,i}_{ts}=\int_s^t S_{tu} \, dx_u^i \quad , \quad X^{xx,ij}_{ts}=\int_s^t S_{tu} \, dx_u^j \, (\der x^i)_{us}.$$
Besides, as in the previous section, the integer $p$ is picked such that $4\ka p>n$ and in this way, $\cb_{2\ka,p}$ becomes a Banach algebra. We shall then expand our integrals so that they can be extended to the case of an irregular noise.

\smallskip

Assume that $f_i\in\cx_3$, where $\cx_3$ is defined at Hypothesis \ref{def:reg-f}, and similarly to Notation~\ref{not:f-prime}, set
$[f_i''(\varphi)](\xi) = \nabla_2^2 \sigma_i(\xi,\varphi(\xi)).$
If $y\in \cq_{2\ka,p}^\ka$, an elementary Taylor expansion of order 3 yields 
\begin{equation*}
\begin{split}
(\der f_i(y))_{us}
&= (X^{x,j}_{us}y^{x,j}_s) \cdot f_i'(y_s)+(X^{xx,jk}y^{xx,jk}_s) \cdot f_i'(y_s)+a_{us}y_s \cdot f_i'(y_s)+\frac{1}{2} (\der y)^2_{us} \cdot f_i''(y_s)\\
&  \hspace{1cm} +y^\sharp_{us} \cdot f_i'(y_s)+\int_0^1 dr \, r \int_0^1 dr' \, \lc f_i''(y_s+rr'(\der y)_{us})-f_i''(y_s)\rc \cdot (\der y)_{us}^2\\
&= (\der x^j)_{us}\cdot y^{x,j}_s \cdot f'(y_s)+x^{2,(jk)}_{us} \cdot y^{xx,jk}_s  \cdot f_i'(y_s)+a_{us}y_s \cdot f_i'(y_s)+\frac{1}{2} (\der y)^2_{us} \cdot f_i''(y_s)\\
&  \hspace{1cm} +(X^{ax,j}_{us} y^{x,j}_s) \cdot f_i'(y_s)+X^{axx,jk}(y^{xx,jk}_s) \cdot f_i'(y_s)\\
&  \hspace{1cm} +y^\sharp_{us} \cdot f_i'(y_s)+\int_0^1 dr \, r \int_0^1 dr' \, \lc f_i''(y_s+rr'(\der y)_{us})-f_i''(y_s)\rc \cdot (\der y)_{us}^2
\end{split}
\end{equation*}
and thus
\begin{align}\label{decompo-der-f}
&(\der f_i(y))_{us}= (\der x^j)_{us}\cdot y^{x,j}_s \cdot f'(y_s)+x^{2,(jk)}_{us} \cdot y^{xx,jk}_s  \cdot f_i'(y_s)+a_{us}y_s \cdot f_i'(y_s) \notag\\
&  \hspace{6cm}
+\frac{1}{2} (\der x^j)_{us} y^{x,j}_s \cdot (\der x^k)_{us} y^{x,k}_s \cdot f_i''(y_s)+f_i(y)^\sharp_{us}. \notag \\
&= (\der x^j)_{us}\cdot y^{x,j}_s \cdot f_i'(y_s)+x^{2,(jk)}_{us} \cdot \lcl  y^{xx,jk}_s  \cdot f_i'(y_s)+y^{x,j}_s \cdot y^{x,k}_s \cdot f_i''(y_s) \rcl+a_{us}y_s \cdot f_i'(y_s) \notag\\
&  \hspace{11cm}
+f_i(y)^\sharp_{us}, 
\end{align}
where $f_i(y)^\sharp_{ts}$ is a residual term which is long (but easy) to write explicitly, and which can be estimated as
\begin{equation}\label{estim-f-y-sh}
\cn[f_i(y)^{\sharp,1}_{ts};\cb_p] \lesssim \lln t-s \rrn^{3\ka} 
\lp 1+\cn[y;\cq_{2\ka,p}^\ka]^{3} \rp.
\end{equation}
In the calculation that leads to (\ref{decompo-der-f}), we have introduced the operators
$$X^{ax,i}_{ts}=\int_s^t a_{tu} \, dx^i_u \quad , \quad X^{axx,ij}_{ts}=\int_s^t a_{tu} \, dx^j_u \, (\der x^i)_{us},$$
while the notation $x^{2,ij}_{ts}$ stands for the usual Lévy area $x^{2,ij}_{ts}=\int_s^t dx^j_u \, (\der x^i)_{us}$. As for the estimation (\ref{estim-f-y-sh}), it is obtained by means of the following natural hypotheses:
\begin{equation}\label{ana-x-ax-rough-2}
X^{ax,i} \in \cac_2^{\ga+2\ka}(\cl(\cb_{2\ka,p} \times \cb_p, \cb_p)),
\quad\mbox{and}\quad
X^{axx,ij} \in \cac_2^{2\ga+2\ka}(\cl(\cb_{2\ka,p},\cb_p)).
\end{equation}

\smallskip

Now, in order to be able to define $\cj_{ts}(\hat dx \, f(y))$ for irregular processes, inject expression (\ref{decompo-der-f}) into the decomposition
$$\int_s^t S_{tu} \, dx_u \, f(y_u)=X^{x,i}_{ts} (f_i(y_s))+\int_s^t S_{tu} \, dx^i_u \,  (\der f_i(y))_{us},$$
to deduce
\begin{multline}\label{int-heur-ordre-2}
\int_s^t S_{tu} \, dx_u \, f(y_u) =X^{x,i}_{ts}(f_i(y_s))+X^{xx,ij}_{ts}(y^{x,i}_s \cdot  f_j'(y_s))+X^{xa,i}_{ts}(y_s,f_i'(y_s))\\
+X^{xxx,ijk}_{ts}(y^{xx,ij}_s \cdot  f_k'(y_s)+y^{x,i}_s \cdot y^{x,j}_s \cdot f_k''(y_s))+r_{ts}, 
\end{multline}
with $r_{ts}=\int_s^t S_{tu} dx_u^i \, (f_i(y)^\sharp_{us})$. The operator $X^{xxx,ijk}$ is here defined by 
$$X^{xxx,ijk}_{ts}=\int_s^t S_{tu} \, dx^k \, X^{2,ij}_{us},$$
and we associate to this operator the (reasonable) regularity assumption
\begin{equation}\label{ana-x-xxx-rough-2}
X^{xxx,ijk} \in \cac_2^{3\ga}(\cl(\cb_{2\ka,p},\cb_{2\ka,p})) \cap \cac_2^{3\ga}(\cl(\cb_p,\cb_p)).
\end{equation}
According to the considerations of Section \ref{section-rough-heuri}, it only remains to establish that $\delha r$ is a term of order 4, as a process with values in $\cb_p$. Actually, we are going to show that $\delha r \in \cac_3^{\ga+3\ka-n/p}(\cb_p)$. It will then suffice to pick $p$ large enough, so that $\ga+3\ka-n/p >1$.

\smallskip

To compute $\delha r$, one must assume the set of algebraic hypotheses (H1) in Hypothesis \ref{hypo-regu}, and a further algebraic relation for $X^{xxx}$:
\begin{equation}\label{alg-x-xxx-rough-2}
(\delha X^{xxx,ijk})_{tus}=X^{x,k}_{tu} X^{2,ij}_{us}  +X^{xx,jk}_{tu}(\der x^i)_{us}.
\end{equation}
As far as the regularity assumptions are concerned, on top of condition~\eqref{ana-x-xxx-rough-2}, we have to modify a little the set (H2) in Hypothesis \ref{hypo-regu}, which becomes:
\begin{equation}\label{ana-x-x-rough-2-1}
X^{x,i} \in \cac_2^\ga(\cl(\cb_p,\cb_p)) \cap \cac_2^\ga(\cl(\cb_{2\ka,p},\cb_{2\ka,p}))
\end{equation}
\begin{equation}\label{ana-x-x-rough-2-2}
X^{x,i} \in \cac_2^{\ga-n/(2p)}(\cl(\cb_{p/2},\cb_p)) \cap \cac_2^{\ga-n/p}(\cl(\cb_{p/3},\cb_{p}))
\end{equation}
\begin{equation}\label{hypo-regu-x-xx-2}
X^{xx,ij} \in \cac_2^{2\ga}(\cl(\cb_p,\cb_p)) \cap \cac_2^{2\ga}(\cl(\cb_{2\ka,p},\cb_{2\ka,p}))\cap \cac_2^{2\ga-n/(2p)}(\cl(\cb_{p/2},\cb_p)).
\end{equation}
Besides, in order to clarify the presentation of this developpement, we will have recourse to the notation $\approx$ to signify ``congruent to a term of order at least $\ga+3\ka-n/p$", or in other words: for all $h,l \in \cac_3$, 
$h \approx l \Leftrightarrow h-l \in \cac_3^{\ga+3\ka-n/p}(\cb_p)$.
Recall that, with this convention, our aim is to establish that $\delha r \approx 0$.

\smallskip

Going back to (\ref{int-heur-ordre-2}), one has, by (\ref{relation-alg-m-l}),
\bean
\lefteqn{-(\delha r)_{tus}}\\
&=& -X^{x,i}_{tu}(\si_i \cdot \der(f(y))_{us})+(\delha X^{xx,ij})_{tus}(y^{x,i}_s \cdot  f_j'(y_s))-X^{xx,ij}_{tu} \der (y^{x,i} \cdot f_j'(y))_{us}\\
& & \hspace{3cm} +(\delha X^{xxx,ijk})_{tus}(y^{xx,ij}_s \cdot  f_k'(y_s)+y^{x,i}_s \cdot y^{x,j}_s \cdot f_k''(y_s))\\
& & \hspace{3cm} -X^{xxx,ijk}_{tu} \der( y^{xx,ij} \cdot f_k'(y)+y^{x,i} \cdot y^{x,j} \cdot f_k''(y))_{us}\\
& & +(\delha X^{xa,i})_{tus}(y_s,f_i'(y_s))-X^{xa,i}_{tu}((\der y)_{us},f_i'(y_s))-X^{xa,i}_{tu}(y_u,\der(f_i'(y))_{us}).
\eean
One can already notice that the fifth and seventh terms of the last sum have the expected regularity. Thanks to the above algebraic relations, together with the decomposition (\ref{decompo-der-f}), we then deduce
\bean
\lefteqn{-(\delha r)_{tus}}\\
&\approx & -X^{xx,ij}_{tu}((\der y^{x,i})_{us} \cdot f_j'(y_s))-X^{xx,ij}_{tu}(y^{x,i}_u \cdot \si_j \cdot \der(f'(y))_{us})\\
& & +X^{xx,jk}_{tu}\lp (\der x^i)_{us} \cdot \lcl y^{xx,ij}_s \cdot  f_k'(y_s)+y^{x,i}_s \cdot y^{x,j}_s \cdot f_k''(y_s) \rcl\rp\\
& &+X^{xa,i}_{tu}(a_{us}y_s,f_i'(y_s))-X^{xa,i}_{tu}((\der y)_{us},f_i'(y_s))\\
& \approx & -X^{xx,ij}_{tu}(y^{x,\sharp,i}_{us} \cdot \si_j \cdot f'(y_s))-X^{xa,i}_{tu}((\delha y)_{us},f_i'(y_s))\\
& & -X^{xx,ij}_{tu} \lp y^{x,i}_u \cdot \der(f_j'(y))_{us}-(\der x^k)_{us} \cdot y^{x,k}_s \cdot y^{x,i}_s \cdot f_j''(y_s) \rp\\
& \approx & -X^{xx,ij}_{tu} \lp  y^{x,i}_u \cdot \lcl \der(f_j'(y))_{us}-(\der x^k)_{us} \cdot y^{x,k}_s \cdot f''(y_s) \rcl \rp.
\eean
It is finally easy to see that $(\der f_i'(y))_{us}=(\der x^k)_{us}y^{x,k}_s \cdot f_i''(y_s)+f_i'(y)^\sharp_{us}$, where $f_i'(y)^\sharp$ is a term such that
$$\cn[f_i'(y)^\sharp_{ts};\cb_{p/2}] \lesssim \lln t-s \rrn^{2\ka} \cb[y;\cq_{2\ka,p}^\ka]^2,$$
By means of Hypothesis (\ref{hypo-regu-x-xx-2}), this statement enables to conclude $\delha r \approx 0$.

\smallskip

Let us turn now to our main aim, which is an extension of the integral to Hölder processes with Hölder exponents greater than $1/4$. We first formalize the assumption on $X$ into:
\begin{hypothesis}\label{hypo-reg-x-rough-2}
We assume that the process $x$ gives birth to operators $X^{x,i}$, $X^{ax,i}$, $X^{xx,ij}$, $X^{xa,i}$, $X^{axx,ij}$, $X^{xxx,ijk}$, for which the algebraic conditions (H1) in Hypothesis \ref{hypo-regu} and the analytical conditions (\ref{ana-x-x-rough-2-1})-(\ref{hypo-regu-x-xx-2}), (\ref{ana-x-ax-rough-2}), (\ref{ana-x-xxx-rough-2}), are satisfied, for some triplet $(\ga,\ka,p)$ such that
$$\ga \in (1/4,1/3] \ , \ \ka \in (0,1/4) \ , \ p\in \N^\ast \ , \ 4\ka p>n \ , \ \ga+3\ka-n/p >1.$$
\end{hypothesis}

Just as in Section \ref{section-cas-rough}, this hypothesis allows to give a sense to the rough integral at stake here:
\begin{proposition}\label{defi-rig-int-rough-2}
Under Hypothesis \ref{hypo-reg-x-rough-2} and assuming that $f=(f_1,\ldots,f_N)$ with $f_i\in \cx_3$ for $i=1,\dots,N$, we set, for any $y\in \hat{\cq}_{\ka,p}^\ka([0,T])$, $\cj(\hat{d}x \, f(y))=(\id-\Laha \delha)(J)$, where, for all $s<t$,
\begin{multline}
J_{ts}=X^{x,i}_{ts}(f_i(y_s))+X^{xx,ij}_{ts}(y^{x,i}_s  \cdot f'_j(y_s))+X^{xa,i}_{ts}(y_s,f_i'(y_s))\\
+X^{xxx,ijk}_{ts}(y^{xx,ij}_s  \cdot f_k'(y_s)+y^{x,i}_s \cdot y^{x,j}_s  \cdot f_k''(y_s)).
\end{multline}

Then one has:
\begin{enumerate}
\item $\cj(\hat{d}x \, f(y))$ is well-defined and there exists $z\in \cq_{2\ka,p}^\ka([0,T])$ such that $\delha z$ is equal to the increment $\cj(\hat{d}x \, f(y))$. Furthermore, for any $0\le s<t\le T$, the integral $\cj_{ts}(\hat{d}x \, f(y))$ coincides with a Riemann type integral for two regular functions $x$ and $y$.
\item The following estimation holds true
\begin{equation}\label{estim-int-rough-2}
\cn[z;\cq_{2\ka,p}^\ka([0,T])] \leq \cfx \lcl 1+\cn[y;\cac_1^0(\cb_{2\ka,p})]^3+T^\al \cn[y;\cq_{2\ka,p}^\ka]^3 \rcl,
\end{equation}
for some $\al >0$.
\item For all $s<t$, $\cj_{ts}(\hat{d}x \, f(y))=\lim_{|\Delta_{[s,t]}| \to 0} \sum_{(t_k)\in \Delta_{[s,t]}} J_{t_{k+1}t_k}$.
\end{enumerate}
\end{proposition}

\subsection{Resolution of the differential system}
As for Theorem \ref{theo-cas-rough}, 
the (local) resolution of the noisy heat equation of roughness order 3 stems from a standard fixed-point argument based on the estimation (\ref{estim-int-rough-2}):
\begin{theorem}\label{theo-cas-rough-2}
Under Hypothesis \ref{hypo-reg-x-rough-2} and assuming that $f=(f_1,\ldots,f_N)$ with $f_i \in \cx_4$; $i=1,\dots,N$, there exists a times $T>0$ for which the system
\begin{equation}
(\delha y)_{ts}=\cj_{ts}(\hat{d}x \, f(y)) \quad , \quad y_0=\psi \in \cb_p,
\end{equation}
interpreted with Proposition \ref{defi-rig-int-rough-2}, admits a unique solution $y$ in $\cq_{2\ka,p}^\ka([0,T])$. 
\end{theorem}

\smallskip

This abstract theorem can then be applied as in Section \ref{section-appli-mbf-13}, thanks to the construction of a path 
$$\textbf{X}=(X^{x,i}, X^{ax,i}, X^{xx,ij}, X^{xa,i}, X^{axx,ij}, X^{xxx,ijk})$$
from a $\ga$-Hölder process $x=(x^{(1)},\ldots,x^{(d)})$, with $\ga >1/4$. The rough path Hypothesis \ref{hypo-x-rough-path} must simply be enhanced in:
\begin{hypothesis}
Assume that the path $x$ allows to construct two processes $x^2 \in \cac_2^{2\ga}(\R^n \otimes \R^n)$, $x^3 \in \cac_2^{2\ga}(\R^n \otimes \R^n \otimes \R^n)$ such that
$$\der x^2=\der x \otimes \der x \quad , \quad x^2+(x^2)^\ast=\der x \otimes \der x,$$
$$\der x^3=x^2 \otimes \der x +\der x \otimes x^2.$$
\end{hypothesis}
The path $\textbf{X}$ can then be rigourously defined via the transformations (\ref{def-ipp-x-x})-(\ref{def-ipp-x-xx}), together with the additional expressions
$$X^{axx,ij}_{ts}=\int_s^t AS_{tu} x^{2,ij}_{us} \, du,$$
$$X^{xxx,ijk}_{ts}=x^{3,(ijk)}_{ts}+\int_s^t AS_{tu} x^{3,(ijk)}_{us} \, du.$$


\bigskip





\begin{thebibliography}{99}

\bibitem{Ad}
R.~A. Adams.
\newblock {\em Sobolev spaces}.
\newblock Academic Press, 1975.

\bibitem{BJ}
X. Bardina and M. Jolis. 
\newblock Multiple fractional integral with Hurst parameter less than $\frac 12$.  
\newblock \textit{Stochastic Process. Appl.}  \textbf{116}  (2006),  no. 3, 463--479.

\bibitem{BE}
Z. Brzezniak and K. Elworthy.
\newblock Stochastic differential equations on Banach manifolds.  
\newblock {\it Methods Funct. Anal. Topology}  6  (2000),  no. 1, 43--84.

\bibitem{CF}
M. Caruana and P.Friz.
\newblock Partial differential equations driven by rough paths.
\newblock {\em Journal of Differential Equations} Volume 247, Issue 1, 1 July 2009, Pages 140-173.

\bibitem{CFH}
M. Caruana, P. Friz and H. Oberhauser.
\newblock A (rough) pathwise approach to a class of non-linear stochastic partial differential equations
\newblock Preprint {\tt arXiv:0902.3352 [math.AP]} (2009).

\bibitem{CQ}
L. Coutin and Z. Qian.
\newblock Stochastic rough path analysis and fractional Brownian motion.
\newblock {\em Probab. Theory Relat. Fields} 122:108-140, 2002.

\bibitem{Da}R. Dalang. 
\newblock{Extending martingale measure stochastic integral with applications to spatially homogeneous S.P.D.E's},
\newblock {\em Electron. J. Probab.} 4, Paper No.6, 29 p., 1999 (electronic).

\bibitem{DPZ}
G.~Da~Prato and J.~Zabczyk.
\newblock {\em Stochastic equations in infinite dimensions}, volume~44 of {\em
 Encyclopedia of Mathematics and its Applications}.
\newblock Cambridge University Press, Cambridge, 1992.

\bibitem{DT}
A. Deya and S. Tindel.
\newblock  Rough Volterra equations 1: the algebraic integration setting.
\newblock {\it Stoch. and Dyn.} 9(3):437-477, 2009.


\bibitem{DTbis}
\newblock A. Deya and S. Tindel.
\newblock Rough Volterra equations 2: convolutional generalized integrals.
Preprint {\tt arXiv:0810.1824 [math.PR]} (2008).

\bibitem{engel00a}
K.-J. Engel and R.~Nagel.
\newblock {\em One-parameter semigroups for linear evolution equations}, volume
 194 of {\em Graduate Texts in Mathematics}.
\newblock Springer-Verlag, New York, 2000.

\bibitem{fattorini99a}
H.~O. Fattorini.
\newblock {\em Infinite-dimensional optimization and control theory}, volume~62
 of {\em Encyclopedia of Mathematics and its Applications}.
\newblock Cambridge University Press, Cambridge, 1999.

\bibitem{FV}
P. Friz and N. Victoir.
\newblock {\em Multidimensional dimensional processes seen as rough paths}.
\newblock Cambridge University Press, to appear.

\bibitem{rough}
M.~Gubinelli.
\newblock Controlling rough paths.
\newblock {\em Jour. Funct. Anal.} 216:86-140, 2004.

\bibitem{g-ramif}
M.~Gubinelli.
\newblock Ramification of rough paths. 
\newblock Preprint {\tt arXiv:math.CA/0610300} (2006), to appear in J. Diff. Eq.


\bibitem{g-kdv}
M.~Gubinelli.
\newblock Rough solutions for the periodic Korteweg-de Vries equation.
\newblock Preprint {\tt arXiv:math/0610006} (2006).

\bibitem{g-abs}
M.~Gubinelli.
\newblock Abstract integration, Combinatorics of Trees and Differential Equations.
\newblock Preprint {\tt arXiv:0809.1821} (2008). To appear in the Proceedings of the Conference on Combinatorics and Physics, MPI Bonn, 2007.

\bibitem{glt}
M.~Gubinelli, A.~Lejay and S.~Tindel.
\newblock Young integrals and SPDEs
\newblock {\em Pot. Anal.} 25:307--326, 2006.

\bibitem{GT}
M.~Gubinelli and S.~Tindel.
\newblock Rough evolution equations.
\newblock Preprint {\tt arXiv:0803.0552 [math.PR]} (2008), to appear in Ann. Prob.

\bibitem{hir}
F. Hirsch.
\newblock Lipschitz functions and fractional Sobolev spaces.
\newblock {Potential Anal.}, 11 (1999), 415-429.

\bibitem{Lej}
A. Lejay.
\newblock An introduction to rough paths.
\newblock In {\em S{\'e}minaire de Probabilit{\'e}s 37}, volume 1832 of
 {\em Lecture Notes in Mathematics}, pages 1--59. Springer-Verlag Heidelberg,
 2003.

\bibitem{Ledoux-Lyons}
M.~Ledoux, T.~Lyons, and Z.~Qian.
\newblock L\'evy area of {W}iener processes in Banach spaces.
\newblock {\em Ann. Probab.}, 30(2):546--578, 2002.

\bibitem{LS}
J. Le\'on and J. San Martin.
\newblock Linear stochastic differential equations driven by a fractional Brownian motion with Hurst parameter less than $1/2$.
\newblock To appear in {\em Stoch. And Stoch. Reports}.

\bibitem{LR}
S. Lototsky and B. Rozovsky.
\newblock Wiener Chaos Solutions of Linear Stochastic Evolution Equations. 
\newblock {\em Annals of Probability} 34, 2006.

\bibitem{LyonsBook}
T.~Lyons and Z.~Qian.
\newblock {\em System control and rough paths}.
\newblock Oxford University Press, 2002.

\bibitem{Lyons}
T.~ Lyons.
\newblock Differential equations driven by rough signals.
\newblock {\em Rev. Mat. Iberoamericana}, 14(2):215--310, 1998.

\bibitem{Maslo-Nualart}
B. Maslowski and D. Nualart.
\newblock Evolution equations driven by a fractional {B}rownian motion.
\newblock {\em J. Funct. Anal.}, 202(1):277--305, 2003.

\bibitem{NNT}
A. Neueunkirch, I. Nourdin, S. Tindel:
\newblock Delay equations driven by rough paths.
\newblock {\em Elec. J. Probab.},  13:2031--2068, 2008.

\bibitem{Nbk}
D. Nualart.
\newblock {\em Malliavin calculus and related topics}.
\newblock Springer, 1995.


\bibitem{PT}
V. P\'erez-Abreu and C. Tudor.
\newblock Transfer principle for stochastic fractional integral.
\newblock {\em Bol. Soc. Mat. Mexicana} 8:55-71, 2002.

\bibitem{PZ}
S. Peszat and J. Zabczyk: 
Nonlinear stochastic wave and heat equations.  
{\it Probab. Theory Related Fields}  116  (2000),  no. 3, 421--443.

\bibitem{paz}
A.~Pazy.
\newblock {\em Semigroups of linear operators and applications to partial
 differential equations}, volume~44 of {\em Applied Mathematical Sciences}.
\newblock Springer-Verlag, New York, 1983.

\bibitem{qt}
Ll. Quer and S. Tindel.
\newblock 
The 1-d stochastic wave equation driven by a fractional Brownian motion.
\newblock {\em Stoch. Processes  Appl.}  117(10):1448--1472,  2007.


\bibitem{Ste}
E. M. Stein.
\newblock {\em Singular integrals and differentiabilty properties of functions}
\newblock Princeton University Press, Princeton, New Jersey, 1970.

\bibitem{stri}
R. Strichartz.
\newblock
Multipliers on fractional Sobolev spaces.
\newblock {\em J. Math. Mech.}, 16 (1967), 1031-1060.

\bibitem{Tei}
J. Teichmann.
\newblock Another approach to some rough and stochastic partial differential esquations.
\newblock Preprint {\tt 	arXiv:0908.2814 [math.PR]}  (2009).

\bibitem{TT}
S. Tindel and I. Torrecilla.
\newblock Some differential systems driven by a fBm with Hurst parameter greater than 1/4
\newblock Preprint {\tt 	arXiv:0901.2010 [math.PR]} (2009).

\bibitem{TTV}
S.~Tindel, C.~A. Tudor, and F.~Viens.
\newblock Stochastic evolution equations with fractional {B}rownian motion.
\newblock {\em Probab. Theory Related Fields}, 127(2):186--204, 2003.

\bibitem{TU}
S. Tindel and J. Unterberger.
\newblock The rough path associated to the multidimensional analytic fBm with any Hurst parameter.
\newblock Preprint {\tt 	arXiv:0810.1408 [math.PR]} (2008).

\bibitem{Un} J. Unterberger.
\newblock
Stochastic calculus for fractional
Brownian motion with Hurst exponent $H>1/4$: a rough path method by 
analytic extension. 
\newblock To appear in {\it Ann. Prob.}

\bibitem{Wa}
J.~B. Walsh.
\newblock An introduction to stochastic partial differential equations.
\newblock In {\em \'Ecole d'\'et\'e de probabilit\'es de Saint-Flour,
 XIV---1984}, volume 1180 of {\em Lecture Notes in Math.}, pages 265--439.
 Springer, Berlin, 1986.

\bibitem{Young}
L.~C. Young.
\newblock An inequality of H\"older type, connected with Stieljes integration.
\newblock {\em Acta Math.}, 67:251--282, 1936.

\bibitem{Za}
M.~Z{\"a}hle.
\newblock Integration with respect to fractal functions and stochastic
 calculus. {I}.
\newblock {\em Probab. Theory Related Fields}, 111(3):333--374, 1998.

\end{thebibliography}
\end{document}